\documentclass[a4paper,11pt,english]{smfart}
\usepackage{amsfonts}
\usepackage{amsmath,amsthm} 
\usepackage{hyperref} 
\usepackage{latexsym}
\usepackage{array}
\usepackage{amssymb}
\usepackage{enumerate}
\usepackage[francais,english]{babel}
\usepackage{color}

\theoremstyle{plain}  \newtheorem{theorem}{Theorem}[section]
\newtheorem{lemma}[theorem]{Lemma}
\newtheorem{proposition}[theorem]{Proposition}
 
\newtheorem{definition}[theorem]{Definition} \theoremstyle{remark}
\newtheorem{remark}[theorem]{Remark}

\newcommand{\lsim}{  \lesssim   }
\newcommand{\gsim}{  \gtrsim   }

\newcommand{\de}{  \delta   }

\renewcommand{\(}{   \left[\hspace{-1ex}\left[\hspace{0.5ex}  }
\renewcommand{\)}{  \hspace{0.5ex} \right]\hspace{-1ex}\right] }
\newcommand{\R}{  \mathbb{R}   }

\newcommand{\eps}{\varepsilon}

\newcommand{\C}{  \mathbb{C}   }
\newcommand{\Z}{  \mathbb{Z}   }
\newcommand{\N}{  \mathbb{N}   }
\newcommand{\Nd}{  \mathcal{N}_2   }
\newcommand{\Nb}{  \bar{\mathbb{N}}   }
\newcommand{\A}{  \mathcal{A}   }
\newcommand{\M}{  \mathcal{M}   }

\newcommand{\Ca}{  \mathcal{C}   }

\newcommand{\E}{  \mathcal{E}   }
\newcommand{\NF}{  \mathcal{NF}   }

\renewcommand{\O}{  \mathcal{O}   }

\newcommand{\T}{  \mathbb{T}   }
\newcommand{\Tc}{  \mathcal{T}   }
\newcommand{\D}{  \mathcal{D}   }

\newcommand{\Cc}{  \mathcal{C}   }
\newcommand{\dd}{  \text{d}   }

\newcommand{\om}{  \omega   }
\newcommand{\Om}{  \Omega   }

\renewcommand{\a}{  \alpha   }
\newcommand{\p}{  \partial   }
\renewcommand{\b}{  \beta   }
\renewcommand{\P}{  \mathcal{P}   }
\newcommand{\ga}{\gamma   }
\newcommand{\s}{  \sigma   }
\newcommand{\lan}{  \langle  }
\newcommand{\ran}{  \rangle  }

\newcommand{\ka}{  \kappa   }
\renewcommand{\r}{  \rho   }
\newcommand{\bk}{  {\bf k}   }
\newcommand{\bj}{  {\bf j}   }

\newcommand{\la}{  \lambda_a   }
\newcommand{\lb}{  \lambda_b   }

\newcommand{\G}{  \Gamma   }

\newcommand{\zz}{  \mathfrak z  }
\renewcommand{\phi}{  \varphi  }
\renewcommand{\L}{  \mathcal{L}   }
\renewcommand{\S}{  \mathbb{S}   }
\newcommand{\nazz}{ (\nabla_\rho\cdot {\mathfrak z} ) }
\newcommand{\diag}{\operatorname{diag}}
\newcommand{\meas}{\operatorname{meas}}
\newcommand{\card}{\operatorname{card}}

\newcommand{\Span}{\operatorname{Span}}
\newcommand{\be}{\begin{equation}}
\newcommand{\ee}{\end{equation}}
\newcommand{\ben}{\begin{equation*}}
\newcommand{\een}{\end{equation*}}
\newcommand{\ban}{\begin{align*}}
\newcommand{\ean}{\end{align*}}
\numberwithin{equation}{section}

\newcommand{\ff}{ [f]^{s}_{\s,\mu,\D} }
\newcommand{\fb}{ [f]^{s,\b}_{\s,\mu,\D} }
\newcommand{\fbp}{ [f]^{s,\b+}_{\s,\mu,\D} }
\newcommand{\Tbp}{\mathcal{T}^{s,\b+}(\s,\mu,\D)}
\newcommand{\Tb}{\mathcal{T}^{s,\b}(\s,\mu,\D)}

\newcommand{\hh}{ [h]^{s}_{\s,\mu,\D} }

\newcommand{\hhh}{ [h]^{s,\b}_{\s,\mu,\D} }

 \author{ Beno\^it Gr\'ebert}
\address{Laboratoire de Math\'ematiques Jean Leray, Universit\'e de Nantes, UMR CNRS 6629\\
2, rue de la Houssini\`ere \\
44322 Nantes Cedex 03, France}
\email{benoit.grebert@univ-nantes.fr}

\title[KAM for  KG on $\mathbb S^2$]
{KAM for for KG on $\mathbb S^2$ and for the quantum harmonic oscillator on $\R^2$.}

\begin{document}

\begin{abstract}
In this paper we prove an abstract KAM theorem adapted to the Klein Gordon equation on the sphere $\mathbb S^2$ and for the  quantum harmonic oscillator on $\R^2$ with regularizing nonlinearity.
  \begin{center} {\bf \large october 2014} \end{center}\end{abstract}

\subjclass{ }
\keywords{Quantum harmonic oscillator, Hartree  equation, KG equation, KAM theory.}
\thanks{
}

\maketitle
\tableofcontents
\section{Introduction.}

If the KAM theorem is now well documented for nonlinear Hamiltonian PDEs in 1-dimensional context (see \cite{kuk87,Kuk2,pos89}) only few results exist for multidimensional PDEs.

Existence of quasi-periodic solutions of space-multidimensional  PDE were first proved in \cite{B1} (see also 
  \cite{B2}) but with a technic based on the Nash-Moser thorem that do not allow to analyse the linear stability of the obtained solutions.
 Some KAM-theorems for small-amplitude solutions of multidimensional beam equations (see \eqref{beam}   above)
  with typical $m$   were obtained in
  \cite{GY05, GY06}. Both works treat equations  with a constant-coefficient nonlinearity 
  $g(x,u)=g(u)$, which  is significantly easier than the general case. 
  The first complete KAM theorem for space-multidimensional PDE was obtained in \cite{EK10}. Also see \cite{Ber1, Ber2}. \\ 
  The technics  developed by Eliasson-Kuksin has been improved in \cite{EGK1,EGK2} to allow a KAM result without external parameters.  In these two papers the authors prove the existence of small amplitude quasi-periodic solutions of the beam equation on the d-dimensional torus. They further investigate  the stability of these solutions and give explicit examples where the solution is linearly unstable and thus exhibits hyperbolic features (a sort of whiskered torus).\\
All these examples concern PDEs on the torus, essentially because in that case the corresponding linear PDE is diagonalized in the Fourier basis and the structure of the resonant sets is almost the same for NLS, NLW or beam equation. In the present paper, adapting the technics in \cite{EK10}, we consider two important examples that do not fit in the Fourier context: the Klein-Gordon equation on the sphere $\mathbb S^2$ and the quantum harmonic oscillator on $\R^2$.\\
Notice that  existence of quasi-periodic solutions for NLW and NLS on compact Lie groups via  Nash Moser technics (and thus without linear stability) has been proved recently in \cite{BP,BCP}.

\medskip

To understand the new difficulties, let us begin by recalling briefly part of the method developed in \cite{EK10}.  Consider the non linear Schr\"odinger equation on $\T^d$
$$iu_t=-\Delta u + \text{nonlinear terms}, \quad x\in \T^d,\ t\in\R.$$
In Fourier variables it reads\footnote{The space $\Z^d$ is equipped with standard euclidian norm: $|k|^2=k_1^2+\cdots k_d^2$.}
$$i \, \dot u_k= |k|^2 u_k +\text{nonlinear terms},\quad k\in\Z^d.$$
So two Fourier modes indexed by $k,j\in\Z^d$ are (linearly) resonant when $|k|^2=|j|^2$. For the beam equation on the torus the relation of resonance is the same. The resonant sets $\E_k=\{j\in\Z^d\mid |j|^2=|k|^2\}$ realize a  natural clustering of $\Z^d$. All the modes in the block $\E_k$ have the same energy  and we can expect that the interactions are small between different blocks but could be of order one inside a block.
With this idea in mind, the principal step of the KAM theory, the resolution of the so called homological equation, leads to the inversion of an infinite matrix which is block-diagonal with respect to this clustering. 
 It turns out that these blocks have cardinality growing with $|k|$ making harder the control of the inverse of this matrix. As a consequence you loose regularity each time you solve the homological equation. Of course this is not acceptable for an infinite induction. The very nice idea in \cite{EK10} consists in considering a sub-clustering constructed as the equivalence classes of the equivalence relation on $\Z^d$ generated by the pre-equivalence relation
$$ a\sim b \Longleftrightarrow \left\{\begin{array}{l} |a|=|b| \\   {|a-b|}
 \leq \Delta \end{array}\right.$$
Let $[a]_\Delta$ denote the equivalence class of $a$. 
The crucial fact (proved \cite{EK10}) is that the blocks are {\it finite}
with a maximal ``diameter''
$$\max_{[a]_\Delta=[b]_\Delta} |a-b|\leq C_d \Delta^{\frac{(d+1)!}2} $$
depending only on $\Delta$. With such clustering, you do not loose regularity when solving the homological equation. 
Further, working in a phase space of analytic functions $u$ or equivalently, exponentially decreasing Fourier coefficients $u_k$, it turns out that the homological equation is "almost" block diagonal relatively to this clustering. 
Then you growth the parameter $\Delta$ at each step of the KAM iteration. 

\medskip

Unfortunately this estimate of the diameter of a block $[a]_\Delta$ by a constant independent of $|a|$ is a sort of miracle that do not persist in other cases. For instance if we consider the quantum harmonic oscillator on $\R^2$
$$i\, u_t=-\Delta u +|x|^2u+\text{nonlinear terms}, \quad x\in\R^2$$
the linear part diagonalizes on a Hermite basis $h_j\otimes h_k$ (see section 3) and the natural clustering is given by
the resonant sets $\{(k,j)\in\N^2\mid k+j= \text{const} \}$. We can easily convince ourself that there is no simple way to construct sub-clustering, compatible with the equation, in such a way the size of the block does no more depend on the energy.\\
So we have to invent a new way to proceed. First we  consider a phase space $Y_s$ with polynomial decay on the Fourier coefficient (corresponding to Sobolev regularity for u) instead of exponential decay and we use a different norm on the finite matrix, namely the Hilbert-Schmidt norm. This technical changes makes disappear the loss of regularity in the resolution of the homological equation. Nevertheless this is not the end of the story since this Sobolev structure of the phase space $\Tc^{s,\b}$ (see section \ref{2}) is not stable by Poisson bracket and thus is not adapted to an iterative scheme. So the second ingredient consists in using a trick previously used in \cite{GT}: we take advantage of the regularizing effect of the homological equation to obtain a solution in a slightly more regular space $\Tc^{s,\b+}$ and then we verify that $\{\Tc^{s,\b},\Tc^{s,\b+}\}\in \Tc^{s,\b}$ (see section 4) which makes possible an iterative procedure. The last problem is to verify that the non linear term, says $P$, belongs to the class $\Tc^{s,\b}$ which imposes a decreasing conditions on the Hilbert-Schmidt norm of the blocks of the Hessian of $P$. Unfortunately, this condition leads to a restriction to the dimension 2 for the Klein Gordon equation on the sphere and impose to consider only regularizing non linearity in the case of the quantum harmonic operator on $\R^2$.\\
In this paper we only consider PDEs with external parameters (similar to a convolution potential in the case of NLS on the torus). Following \cite{EGK2} we could expect to remove these external parameters (and to use only internal parameters) but the technical cost would be very high.

\medskip

We now state the result that we obtain for the Klein Gordon equation.
Denote by $\Delta$ the Laplace-Beltrami
operator on the sphere $\mathbb S^2$ and let $\Lambda_0=(-\Delta +m)^{1/2}$.   
The spectrum of $\Lambda_0$ equals
$\{\sqrt{j(j+1)+m}\mid\ j\geq 0\}.$
For each $j\geq 1$ let $E_j$ be the associated eigenspace, its dimension is $ 2j+1$. We denote by $\Psi_{j,l}$  the standard  harmonic function of degree $j$ and order $\ell$ so that we have
$$E_j=\Span\{\Psi_{j,l},\  l=-j,\cdots,j\}.$$
We denote
$$\E:=\{(j,\ell)\in\N\times\Z\mid j\geq 0\text{ and }\ell=-j,\cdots,j\}$$
in such a way that $\{\Psi_{a},\, a\in\E\}$ is a basis of $L_2(\S^2, \C)$.\\
We introduce the harmonic multiplier $M_\r$
defined on the basis $(\Psi_a)_{a\in\E}$ of $L^2(\S^2)$ by
\be\label{MKG}
M_\r \Psi_a=\r_a\Psi_a\quad \text{ for } a\in\E
\ee
where $(\r_a)_{a\in\E}$ is a bounded sequence of  nonnegative real numbers.

Let $g$ be a real analytic function on $\mathbb S^2\times \R$  such that $g$ vanishes at
least at order 2 in the second variable at the origin. We consider the following 
 nonlinear Klein-Gordon equation 
\be
\label{KG} 
(\partial_{t}^2-\Delta+m+\de M_\r)u=\eps g(x,u),\quad t\in\R,\ x\in\S^2 
\ee
where $\de>0$ and $\eps>0$ are   small parameters.\\
Introducing $\Lambda=(-\Delta +m+\de M_\r)^{1/2}$ and $v=u_t\equiv\dot u$, \eqref{KG} reads
\ben \left\{\begin{array}{ll}
 \dot u &= - 
 v,\\
 \dot v &=\Lambda^2 u    +\eps g(x,u).
\end{array}\right. 
\een
Defining 
 $
 \psi =\frac 1{\sqrt 2}(\Lambda^{1/2}u  + i\Lambda^{-1/2}v) $
 we get
$$
\frac 1 i \dot \psi =\Lambda \psi+ \frac{1}{\sqrt 2}\Lambda^{-1/2}g\left(x,\Lambda^{-1/2}\left(\frac{\psi+
\bar\psi}{\sqrt 2}\right)\right)\,.
$$
Thus, if we endow the space   $L_2(\S^2, \C)$ with the standard  real symplectic structure given by the two-form
$\ 
-id\psi\wedge d\bar \psi =- du\wedge dv,$
 then equation 
 \eqref{KG} becomes a Hamiltonian system 
$$\dot \psi=i\frac{\partial H}{\partial \bar\psi}$$
with the hamiltonian function
$$
H(\psi,\bar\psi)=\int_{\S^2}(\Lambda \psi)\bar\psi \dd x +\eps\int_{\S^2}G\left(x,\Lambda^{-1/2}\left(\frac{\psi+\bar\psi}{\sqrt 2}\right)\right)\dd x.
$$
where $G$ is a primitive of $g$ with respect to the variable $u$: $g=\partial_u G$.\\
The linear operator $\Lambda$ is diagonal in the  basis  $\{\Psi_{a}, \, a\in\E\}$:
$$
\Lambda \Psi_{a}=\la \Psi_a,\;\;\la= \sqrt{w_a(w_a+1)+m+\de\r_a},
\qquad  \forall\,a\in\E
$$
 where we set
$$w_{(j,\ell)}=j\quad \forall\,(j,\ell)\in\E.$$
Let us decompose $\psi$ and $\bar\psi$  in the   basis $\{\Psi_a, \, a\in\E\}$:
$$
\psi=\sum_{a\in\E}\xi_a \Psi_a,\quad \bar\psi=\sum_{a\in\E}\eta_a \Psi_{a}\,.
$$
On $\P_\C:=\ell^2(\E,\C)\times\ell^2(\E,\C)$ endowed with the  complex 
symplectic structure
${ -}i\sum_s \dd\xi_s\wedge\dd\eta_s$ we consider the Hamiltonian system
\be \label{KG2} \left\{\begin{array}{ll}\dot \xi_a&=i\frac{\partial H}{\partial \eta_a}\\ \dot \eta_a&=-i\frac{\partial H}{\partial \xi_a}\end{array}\right. \quad a\in\E\ee
where the Hamiltonian function $H$ is given by
\be\label{HKG}H=\sum_{a\in\E}\la \xi_a\eta_a+\eps\int_{\S^2}G\left(x,\sum_{a\in\E}\frac{(\xi_a+\eta_a)\Psi_a}{\sqrt{ 2}\ \la^{1/2}}\right)\dd x.\ee
The Klein Gordon equation \eqref{KG} is then  equivalent to the  Hamiltonian system \eqref{KG2} restricted to the real subspace 
$$\P_\R:=\{(\xi,\eta)\in \ell^2(\E,\C)\times\ell^2(\E,\C)\mid \eta_a=\bar\xi_a, \ a\in\E\}.$$
Let $\A\subset \E$ a finite subset of cardinal $n$ satisfying the {\it admissibility} condition 
\be\label{admissible}\A\ni(j_1,\ell_1)\neq(j_2,\ell_2)\in\A \Rightarrow j_1\neq j_2.\ee
We fix $I_a\in[1,2]$ for $a\in\A$, the initial $n$ actions, and we write the modes $\A$ in action-angle variables:
$$\xi_a=\sqrt{I_a+r_a}e^{i\theta_a},\quad \eta_a=\sqrt{I_a+r_a}e^{-i\theta_a}.$$
We define $\L=\E\setminus\A$ 
and, to simplify the presentation, we assume  that
$$\r_{j,l}=\r_j \text{ for } (j,\ell)\in \A\ ;\ m_{j,l}=0 \text{ for } (j,\ell)\in \L.$$
Set
\begin{align*}
w_{j,\ell}&=j\quad \text{ for }(j,\ell)\in\E,\\
\lambda_{j,\ell}&= \sqrt{j(j+1)+m}\text{ for }(j,\ell)\in\L,\\
(\om_0)_{j,\ell}(\r)&=\sqrt{j(j+1)+m+\de\r_j}\text{ for }(j,\ell)\in\A,\\
\zeta&=(\xi_a,\eta_a)_{a\in\L}.
\end{align*}
With this notation $H$ reads (up to a constant)
$$H(r,\theta,\zeta)= \lan\om_0(\r),r\ran+\sum_{a\in\L} \la\xi_a\eta_a+ \eps f(r,\theta,\zeta)$$
where 
$$f(r,\theta,\zeta)=
\int_{\S^2}G\left(x, \hat u(r,\theta,\zeta)(x)\right)dx$$
and 
\be\label{uhat}
\hat u(r,\theta,\zeta)(x)=\sum_{a\in\A}\frac{\sqrt{I_a+r_a}\cos{\theta_a}}{\la^{1/2}}\Psi_a(x)+\sum_{a\in\L}\frac{(\xi_a+\eta_a)}{\sqrt{ 2}\ \la^{1/2}}  \Psi_a(x) .\ee

 Let us set $ u_1(\theta,x) = \hat u(0,\theta;0)(x) $.
%=u_1(x;I,\theta)$ as \\
%$u_1=  \sum_{a\in\A}   ( {\sqrt{ 2\la} )^{-1/2}  {\sqrt{I_a}\big( e^{i\theta_a}\phi_a(x)+e^{-i\theta_a}\phi_{-a}(x)} } \big).\,$
 Then  for any $I\in[1,2]^n$ and $\theta_0\in\T^n$ 
 the function $(t,x)\mapsto u_1(\theta_0+t\om,x)$ is a quasi periodic solution of \eqref{KG} with
 $\eps=0$.  Our main theorem 
 states that for most external
 parameter $\r$ this quasi-periodic solution persists (but is  sightly deformed) when we turn on the nonlinearity.
\begin{theorem}\label{thmKG}
For $\eps$ sufficiently small (depending on $n$, $s$ and $g$) and  satisfying \footnote{The coefficient 12 is of course non optimal and we note in remark \ref{rem-m=0} that when $m=0$ 12 can be replace by 4.}
$$\eps\leq \Big(\frac{\de}{4\max({w_a,\ a\in\A})}\Big)^{12}$$
 there exists a Borel subset
$$\D'\subset [1,2]^n, \quad
\meas([1,2]^n\setminus\D')\leq C\eps^{\a},$$ 
such that for  $\rho\in\D'$, there is a function 
$ u(\theta,x)$, analytic in $\theta\in\T^n_{\frac\s 2}$ and smooth in $x\in\S^2$, satisfying 
$$\sup_{|\Im\theta|<\frac\s 2}\|u(\theta,\cdot)-u_1(\theta,\cdot)\|_{H^{s}(\S^2)}
\leq \eps^{1/6},$$
and there is a mapping  $$\om':\D'\to \R^n,\quad
 \|\om'-\om\|_{C^1(\D')}\leq \eps^{1/6},$$
such that for any $\r\in \D'$ the function 
$$
u(t,x)=u(\theta+t\om'(\r),x)
$$
is a   solution of the Klein Gordon equation \eqref{KG}. Furthermore this solution  is linearly stable.\\ 
 The positive 
  constant $\a$  depends only on  $n$ while $C$ also depends on $g$ and $s$.
 \end{theorem}

We will deduce Theorem \ref{thmKG} from an abstract KAM result stated in section 2 and proved in section 6. The application to the quantum harmonic oscillator is detailed in section 3.2.\\
In section 4 we study the Hamiltonian flows generated by Hamiltonian functions in $\Tc^{s,\beta}$. In  section 5 we detail the resolution of the homological equation. In both sections 4 and 5 we use   technics and proofs that were developed in \cite{EK10} and \cite{EGK1}. The novelty is the use of Hilbert-Schmidt norm on the matrix and the use of two different class of Hamiltonians: $\Tc^{s,\beta}$ and $\Tc^{s,\beta+}$.  For convenience of the reader we repeat all the arguments.

%General remark: since the second Melnikov condition is satisfied without any restriction on  $|a-b|$, we do not need exponential decay on the off diagonal term of $\nabla^2_\zeta P$. So we can consider a phase space $Y_s$ with polynomial decay instead of exponential decay.

\section{Setting and abstract KAM theorem.}\label{2}
\noindent {\bf Notations.} In this section we state a KAM result for a Hamiltonian $H=h+\eps f$ of the following form
$$
H= \lan\om(\r), r \ran+\frac 1 2 \langle \zeta,  A(\r)\zeta\rangle +  f(r,\theta,\zeta;\r)$$
where 
\begin{itemize}
\item $\om \in \R^n$ is the frequencies vector corresponding to the internal modes in action-angle variables $(r,\theta)\in\R^n_+\times \T^n$.
\item  $\zeta=(\zeta_s)_{s\in\L}$ are the external modes:    $\L$ is an infinite set of indices, $\zeta_s=(p_s,q_s)\in\R^2$ and $\R^2$ is endowed with the standard symplectic structure $dq\wedge dp$.
\item $A$ is a linear operator acting on the external modes, typically $A$ is diagonal.
\item  $f$ is a perturbative Hamiltonian depending on all the modes and is of order $\eps$ where
 $\eps$ is a small parameter.
\item  $\r$ is an external parameter in $\D$ a compact subset of $\R^p$ with $p\geq n$.
\end{itemize}
We now detail the structure beyond these objects and the hypothesis needed for the KAM result.\\

\noindent {\bf Cluster structure on $\L$.}
Let $\L$ be a set of indices and   $w:\L\to \N$ be an "energy" function\footnote{We could replace the assumption that $w$ takes integer values by $\{w_a-w_b\mid a,b\in\L\}$ accumulates on a discrete set. }
 on $\L$.
We consider the clustering of $\L$ given by 
$\L=\cup_{a\in\L}[a]$ associated to equivalence relation 
$$b\sim a \Longleftrightarrow w_a=w_b.$$
We denote  $\hat\L=\L/\sim$.
% and we assume that, for each $M>0$  there exists ${\a_1}>0$ and $C>0$
%\be\label{Nenergy}
%\card\{a\in\hat\L\mid w_a\leq M\}\leq M^{\a_1}
%\ee
We  assume that the cardinal of each energy level is finite and  that there exist $C>0$ and  $d>0$ two constants such that the cardinality of $[a]$ is controlled by $Cw_a^{d}$:
\be\label{block}d_a=d_{[a]}=\card\{b\in\L\mid w_b=w_a\}\leq Cw_a^{d}.\ee

 \smallskip

\noindent {\bf Linear space.}
Let  $s\geq 0$, we consider the complex  weighted $\ell_2$-space
$$
Y_s=\{\zeta=(\zeta_a\in\C^2,\ a\in \L)\mid \|\zeta\|_s<\infty\} 
$$
where\footnote{We provide $\C^2$ with the euclidian norm, $|\zeta_a|=|(p_a,q_a)|=\sqrt{|p_a|^2+|q_a|^2}$.}
$$
\|\zeta\|_s^2=\sum_{a\in\L}|\zeta_a|^2 w_a^{2s}.$$
We also introduce for\footnote{The constraint $\b\leq1$ is technically convenient but not necessary.} $1\geq\b\geq 0$ the complex  weighted $\ell_\infty$-space
$$
L_\b=\{\zeta=(\zeta_a\in\C^2,\ a\in \L)\mid |\zeta|_\b<\infty\} $$
where
$$
|\zeta|_\b=\sup_{a\in\L}|\zeta_{[a]}| w_a^{\b},\quad |\zeta_{[a]}|^2=\sum_{b\in[a]}|\zeta_b|^2.$$
We note that if $s\geq \b$ then $Y_{s}\subset  L_\b$.

In the spaces $Y_s$ acts the linear operator $J$,
$$
J\ :\ \{\zeta_a\}\mapsto \{\s_2\zeta_a\}, \quad \text{with } \s_2=\left(\begin{array}{cc} 0&-1\\1&0\end{array}\right).$$
 It provides the spaces $Y_s$, $s\geq 0$, with the symplectic structure $J\dd\zeta\wedge\dd \zeta$. To any $C^1$-smooth function defined on a domain $\O\subset Y_s$,  corresponds the Hamiltonian equation 
 $$
 \dot \zeta =J\nabla f(\zeta),$$
 where $\nabla f$ is the gradient with respect to the scalar product in $Y$.
 
 \smallskip
 
  \noindent {\bf Infinite matrices.}
 We denote by $\M$ the set of infinite matrix $A:\L\times \L\to \M_{2\times 2}$ with value in the space of real $2\times 2$ matrices that are symmetric
$$
A_s^{s'}=A_{s'}^s,\quad \forall s,\ s'\in \L$$
and satisfy
$$
|A| := \sup_{a,b\in \L}\left\|A_{[a]}^{[b]}\right\|_{HS}<\infty$$
where $A_{[a]}^{[b]}$ denotes the restriction of $A$ to the block $[a]\times[b]$ and $\|\cdot\|_{HS}$ denotes the Hilbert Schmidt norm:
$$\|M\|_{HS}^2:=\sum_{j,\ell}|M_{j\ell}|^2.$$
The for $\beta\geq 0$ we define $\M_\beta$ the subset of $\M$ such that
$$
|A|_\beta := \sup_{a,b\in \L}w_a^{\beta}w_b^{\beta}\|A_{[a]}^{[b]}\|_{HS}<\infty .$$

\noindent {\bf A class of Hamiltonian functions.}
Let us fix any $n\in\N$. 
On the space
$$
\C^n\times \C^n \times Y_s$$
we define the norm
$$\|(z,r,\zeta)\|_s=\max( |z|, |r|, \|\zeta\|_s).$$
%\marginpar{\color{blue} what norms $|z|$ and  $|r|$?}
For $\sigma>0$ we denote
$$
\T^n_\sigma=\{z\in\C^n: | \Im z|<\sigma\}/2\pi\Z^n.
$$
%\marginpar{$\O_\mu (B)$}$$\O_\mu (B)=\{x\in B\mid \|x\|_B<\mu\}.$$
%
For $\sigma,\mu\in(0,1]$ and $s\ge0$ we set
$$
\O^s(\s,\mu)=\T^n_\s\times \{r\in\C^n: |r|<\mu^2\}\times \{\zeta\in Y_s: \|\zeta\|_s<\mu\}
$$
We will denote points in $\O^s(\s,\mu)$ as $x=(\theta,r,\zeta)$.
A
 function defined on a domain $\O^s(\s,\mu)$, is called {\it real} if it gives real values to real
arguments.\\
Let 
$$
\D=\{\rho\}\subset \R^p 
$$ 
be a compact set of positive Lebesgue  measure. This is the set
of parameters upon which will depend our objects. Differentiability of functions on $\D$
is understood in the sense of Whitney. So $f\in C^1(\D)$ if it may be extended to a $C^1$-smooth
function $\tilde f$ on $ \R^p$, and $|f|_{ C^1(\D)}$ is the infimum of  $|\tilde f|_{ C^1(\R^p)}$,
taken over all $C^1$-extensions $\tilde f$ of $f$. \\
If $(z,r,\zeta)$ are $C^1$ functions on $\D$, then we define
$$\|(z,r,\zeta)\|_{s,\D}=\max_{j=0,1}( |\partial^j_\r  z|, |\partial^j_\r  r|, \|\partial^j_\r \zeta\|_s).$$
Let 
 $f:\O^0(\s,\mu)\times\D\to \C$ be a $\C^1$-function, 
 real holomorphic in the first variable $x$,
  such that for all $\rho\in\D$
 $$
 \O^{s}(\s,\mu)\ni x\mapsto \nabla_\zeta f(x,\rho)\in Y_{s}\cap L_\b$$
and
 $$
 \O^{s}(\s,\mu)\ni x\mapsto\nabla^2_{\zeta} f(x,\rho)\in \M_\beta$$
are real holomorphic functions. We denote this set of functions by 
  $\Tc^{s,\beta}(\s,\mu,\D)$. \\
 For a function $f\in \Tc^{s,\beta}(\s,\mu,\D)$ we define
 the norm 
 $$[f]^{s,\beta}_{\s,\mu,\D}$$
 through 
$$\sup
\max( |\partial^j_\r f(x,\r)|,\mu \|\partial^j_\r \nabla_\zeta f(x,\r)\|_{s},\mu |\partial^j_\r \nabla_\zeta f(x,\r)|_{\b},
\mu^2|\partial^j_\r \nabla^2_\zeta f(x,\r)|_{\beta}),
$$
where the supremum is taken over all
$$
j=0,1,\ x\in O^{\ga}(\s,\mu),\ \rho\in\D.$$
We set $\Tc^{s}(\s,\mu,\D)=\Tc^{s,0}(\s,\mu,\D)$ and $[h]^{s}_{\s,\mu,\D}=[h]^{s,0}_{\s,\mu,\D}$.

\medskip

\noindent{\bf Normal form:} 
We introduce the orthogonal projection $\Pi$ defined on the $2\times 2$ complex matrices 
$$\Pi: \M_{2\times 2}(\C)\to \C I+\C J$$
where 
$$I=\left(\begin{array}{cc} 1&0\\0&1\end{array}\right) \quad \text{and}\quad J=\left(\begin{array}{cc} 0&-1\\1&0\end{array}\right).$$
\begin{definition}
 A matrix $A:\ \L\times \L\to \M_{2\times 2}(\C)$ is on normal form  and we denote $A\in  \NF$ if
 \begin{itemize}
 \item[(i)] $A$ is real valued,
 \item[(ii)] $A$ is symmetric, i.e. $A_b^a={}^t\hspace{-0,1cm}A_a^b$,
 \item[(iii)] $A$ satisfies $\Pi A=A$,
 \item[(iii)] $A$ is block diagonal, i.e. $A_b^a=0$ for all $w_a\neq w_b$.
 \end{itemize}
 \end{definition}
To a real symmetric matrix $A=(A_a^b)\in \M$  we associate in a unique way a real quadratic form on $Y_s\ni (\zeta_a)_{a\in \L}=(p_a,q_a)_{a\in\L}$ 
$$q(\zeta)=\frac 1 2 \sum_{a,b\in\L}\langle \zeta_a,\ A_a^b\zeta_b\rangle.$$
In the complex variables, $z_a=(\xi_a,\eta_a),\ a\in \L$, where
$$
\xi_a=\frac 1 {\sqrt 2} (p_a+iq_a),\quad \eta_a =\frac 1 {\sqrt 2} (p_a-iq_a),$$
we have
$$
q(\zeta)=\frac 1 2\langle \xi,\nabla_\xi^2 q\ \xi\rangle+ \frac 1 2\langle \eta,\nabla_\eta^2 q\ \eta\rangle+\langle \xi,\nabla_{\xi}\nabla_{\eta} q\ \eta\rangle    .$$
The matrices $\nabla_\xi^2 q$ and $\nabla_\eta^2 q$ are symmetric and complex conjugate of each other while $\nabla_{\xi}\nabla_{\eta} q$ is Hermitian. If $A\in \M_\beta$ then
\be\label{q}
\sup_{a,b}\big\|(\nabla_\xi\nabla_\eta q)_{[a]}^{[b]}\big\|_{HS} \leq\frac{|A|_\b}{(w_aw_b)^{\beta}}.\ee
We note that if $A$ is on normal form, then the associated  quadratical form $q(\zeta)=\frac 1 2 \langle \zeta,A\zeta\rangle$ reads in complex variables
\be\label{Q} q(\zeta)=\langle \xi,Q \eta\rangle \ee
where $Q:\L\times\L\to \C$ is 
\begin{itemize}
\item[(i)] Hermitian, i.e. $Q_b^a=\overline{Q_a^b}$,
\item[(ii)] Block-diagonal. 
\end{itemize}
%Furthermore if a matrix $A$ is on normal form, then the corresponding Hamiltonian operator $JA$ has discrete pure imaginary spectrum.
In other words, when $A$ is on normal form, the associated quadatic form reads
$$q(\zeta)=\frac 1 2 \langle p,A_1 p\rangle+ \langle p,A_2 q\rangle+\frac 1 2 \langle p,A_1 q\rangle$$
with $Q=A_1+iA_2$ Hermitian.\\
 By extension we will say that a Hamiltonian is on normal form if it reads
\be\label{h}
h=\lan\om(\r), r \ran+\frac 1 2\langle \zeta, A(\r)\zeta\rangle\ee
with $\om(\r)\in\R^n$ a frequency vector and $A(\r)$ on normal form for all $\r$.

\smallskip

\subsection{Hypothesis on the spectrum of $A_0$.} We assume that  $A_0(\r)$ a real diagonal matrix whose diagonal elements $\la(\r)>0, \ a\in\L$ are $C^1$. Our hypothesis depend on two constants $1>\delta_0>0$ and $c_0>0$ fixed once for all.

\smallskip

\noindent {\bf Hypothesis A1 -- Asymptotics.}
We assume that there exist $\ga\geq1$   such that
%\be\label{lanot0}
%\la(\r)\geq \delta\quad \mbox{ uniformly for  } \r\in\D \text{ and } a\in\L.
%\ee
%%and on the other hand that  for all $\r\in\D$:
%%\be\label{hypoB}
%%B(\r)\in\NF\cap \M_w \text{ and }|B|_{\beta,D}\leq \frac \delta 8.
%%\ee
\be\label{laequiv}
\la(\r )\geq c_0\, w_a^{\ga} \quad \mbox{  for  } \r\in\D \text{ and } a\in\L
\ee
and
\be\label{la-lb}
|\la(\r )-\lb(\r)|\geq {c_0}{|w_a-w_b|} \quad \text{for }a,b\in\L \mbox{ and  for  } \r\in\D
\ee
%\begin{example} \label{ex1}
%Quantum harmonic Oscillator in $\R^d$.\\
%Let $\L\subset \N^d$, $w_a=|a|_{\ell_1}:=|a_1|+\cdots+|a_d|$ and $d^*=d-1$.\\
%$\la=2(a_1+\cdots a_d)-2$ satisfies the asymptotic conditions with $\ga_1=1$ and $\ga_2=0$.
%\end{example}
%\begin{example} \label{ex2}
%Klein Gordon equation on $\mathbb S^d$.\\
%Let $\L:=\{(n,k)\in\N^2\mid 1\leq k\leq n^{{d-1}}\}$, $w_a=a_1$ and $d^*=d-1$.\\
%Let $m>0$, $\la=\sqrt{a_1^2+m}$ satisfies the asymptotic conditions with $\ga_1=1$ and $\ga_2=0$.
%\end{example}

\smallskip

\noindent
{\bf Hypothesis A2 -- non resonances.}
There exists a $\delta_0>0$ such that for all $\Ca^1$-functions
$$\omega:\D_0\to \R^n,\quad |\omega-\omega_0|_{\Ca^1(\D_0)}<\delta_0,$$
the following hold for each  $k\in\Z^n\setminus 0$:
\begin{itemize}
\item[(i)]
either $$ |  \langle k,\om(\r)\rangle|\geq \delta_0$$ for all  $\r\in\D_0$, 
or there exits a unit vector ${\mathfrak z}\in\R^p$  such that 
$$ \nazz  (\langle k,\omega\rangle)  \geq \delta_0$$ for all  $\r\in\D_0$;

\item[(ii)]
%\marginpar{\color{blue} also for $k=0$}
either $$ | \langle k,\om(\r)\rangle+\la(\r)|\geq \delta_0 w_a$$ for all  $\r\in\D_0$ and $a\in\L$ 
or there exits a unit vector ${\mathfrak z}\in\R^p$  such that 
$$ 
\nazz\big(  \langle k,\om(\r)\rangle+\la(\r) \big) \geq \delta_0
$$ 
for all  $\r\in\D_0$ and $a\in\L$;

\item[(iii)]
%\marginpar{\color{blue} also for $k=0$}
either $$ |  \langle k,\om(\r)\rangle+\la(\r)+\lb(\r)|\geq \delta_0(w_a+w_b)$$ for all  $\r\in\D_0$ and 
$a,\ b\in\L$  or there exits a unit vector ${\mathfrak z}\in\R^p$  such that 
$$
  \nazz\big( \langle k,\om(\r)\rangle+\la(\r)+\lb(\r) \big) \geq \delta_0
  $$ for all  $\r\in\D_0$ and $a,\ b\in\L$;

\item[(iv)]
%\marginpar{\color{blue} also for $k=0$}
either 
$$
 |  \langle k,\om(\r)\rangle+\la(\r)-\lb(\r)|\geq \delta_0 (1+|w_a-w_b|)
$$
 for all  $\r\in\D_0$ and $a,\ b\in\L$ 
or there exits a unit vector ${\mathfrak z}\in\R^p$  such that 
$$ \nazz \big( \langle k,\om(\r)\rangle+\la(\r)-\lb(\r)\big)  \geq \delta_0$$ for all  $\r\in\D_0$ and $a,\ b\in\L$
\end{itemize}

The assumption (iv) above will be used to bound from below divisors $|  \langle k,\om(\r)\rangle+\la(\r)-\lb(\r)|$
with $w_a,\, w_b\sim1$. To control the (infinitely many) divisors with $\max(w_a,w_b)\gg1$ we need another assumption:
\medskip

\noindent
{\bf Hypothesis A3 -- second Melnikov condition in measure.} There exist absolute constant $\a_1>0$, $\a_2>0$ and $C>0$ such that
for all $\Ca^1$-functions
$$\omega:\D\to \R^n,\quad |\omega-\omega_0|_{\Ca^1(\D)}<\delta_0,$$
the following holds:\\
for each $\ka>0$ and $N\ge1$ there exists a  closed subset $\D'=\D'(\om_0,\kappa,N)\subset \D$  satisfying 
\be\label{mesomega2}
\meas(\D\setminus {\D'})\leq CN^{\a_1}  (\frac{\ka}{ \delta_0} )^{\a_2}\quad (\a_1,\a_2\ge 0)  \ee
 such that for all $\r\in{\D'}$,  all $0<|k|\leq N$ and  all $a,b\in\L$ 
 we have
 \be\label{D33}
| \langle k,\om(\r)\rangle +\la(\r)-\lb(\r)|\geq \ka(1+|w_a-w_b|).
\ee
%Remark \ref{remro} applies also to Hypothesis A3.

\subsection{The abstract KAM Theorem.}
We are now in position to state our abstract KAM result.
\begin{theorem}\label{main} Assume that 
\be \label{h0}h_0=\lan\om_0(\r), r\ran +\frac 1 2 \langle \zeta,  A_0(\r)\zeta\rangle\ee  with the spectrum of $A_0$ satisfying Hypothesis A1, A2, A3 and let $f\in\Tc^{s,\beta}(\D,\s,\mu)$ with $\b>0$, $s>0$. 
There exists $\eps_0$ (depending on $n,d,s,\b,\s,\mu$ and on $h_0$\footnote{Dependence on $h_0$ means: dependence on $\A$, on $\Cc^1$-norm of $\r\mapsto \om_0(\r)$ and  $\r\mapsto A_0(\r)$.}) such that if 
$$[f]^{s,\beta}_{\s,\mu,\D}=\eps<\min(\eps_0 , \delta_0^4)$$
there is a $\D'\subset \D$ with $\text{meas}(\D\setminus \D')\leq \eps^\a$ such that for all $\r\in \D'$ the following hold: There are a real analytic symplectic diffeomorphism
$$\Phi:\O^s(\s/2,\mu/2)\to \O^s(\s,\mu)$$
and a vector $\om=\om(\r)$ such that 
$$
(h_{0}+  f)\circ \Phi= \lan\om(\r), r\ran +\frac 1 2 \langle \zeta,  A(\r)\zeta\rangle + \tilde f(r,\theta,\zeta;\r)$$
where $\partial_\zeta \tilde f=\partial_r \tilde f=\partial^2_{\zeta\zeta}\tilde f=0$ for $\zeta=r=0$ and $ A:\L\times\L \to \M_{2\times 2}(\R)$  is on normal form, i.e. $\tilde A$ is real symmetric and block diagonal: $ A_{ab}=0$ for all $w_a\neq w_b$.\\
Moreover $\Phi$ satisfies
$$
\|\Phi-Id\|_s\leq  \eps^{1/6}$$
for all $(r,\theta,\zeta)\in \O^s(\s/2,\mu/2)$, and
\begin{align*} &\left|  A(\r)-A_0(\r)\right|_\b\leq  \eps^{1/6}, \\
&| \om(\r)-\om_0(\r)|\leq  \eps^{1/6}
\end{align*}
 for all $\r\in \D'$.\\
 The constant $\a$  only depends on $n$, $d$, $s$, $\b$, $\a_1$, $\a_2$.  
\end{theorem}

This  normal form result has dynamical consequences.  For $\r\in\D'$, the torus $\{0\}\times\T^n\times\{0\}$ is invariant by the flow of $(h_0+f)\circ \Phi$ and the dynamics of the Hamiltonian vector field of $h_0+f$ on the $\Phi(\{0\}\times\T^n\times\{0\})$ is the same as that of $$\lan\om(\r), r\ran +\frac 1 2 \langle \zeta,  A(\r)\zeta\rangle .$$
The Hamiltonian vector field on the torus $\{\zeta=r=0\}$ is
$$
\left\{\begin{array}{l}
\dot\zeta=0\\
\dot \theta=\om\\
\dot r=0,
\end{array}\right.
$$
and  the flow on the torus is linear: $t\mapsto \theta(t)=\theta_0+t\om$.\\
Moreover, the linearized equation on this torus reads 
$$
\left\{\begin{array}{l}
\dot\zeta=JA\zeta+J\partial^2_{r\zeta}f(0,\theta_0+\om t,0)\cdot r\\
\dot\theta=\partial^2_{r\zeta}f(0,\theta_0+\om t,0)\cdot \zeta+\partial^2_{rr}f(0,\theta_0+\om t,0)\cdot r\\
\dot r=0.
\end{array}\right.
$$
Since $A$ is on normal form (and in particular real symmetric and block diagonal) the eigenvalues of the $\zeta$-linear part are purely imaginary: $\pm i \tilde\lambda_a,\ a\in\L$. Therefore the invariant torus is linearly stable in the classical sense (all the eigenvalues of the linearized system are purely imaginary). 
%It turns out that we can prove more: 
% the linearized equation is reducible to constant coefficients. The reducibility is obtained when the
%imaginary part $\tilde\lambda_a$  of the eigenvalues are non-resonant with
%respect to the frequency vector  $\om$, something which can be verified if we restrict (if necessary)  the
%set $\D'$ arbitrarily little. Then the $\zeta$-component (and of
%course also the $r$-component) will have only quasi-periodic
%(in particular bounded) solutions. The $\theta$-component  may have
%a linear growth in $t$, the growth factor (the ``twist'') being
%linear in $\hat r$.

\section{Applications}
\subsection{The Klein Gordon equation on $\mathbb S^2$}
 In this section we prove Theorem \ref{thmKG} as a corollary of Theorem \ref{main}. We recall some notations introduced in the introduction. 
 We denote
$$\E:=\{(j,\ell)\in\N\times\Z\mid j\geq 0\text{ and }\ell=-j,\cdots,j\}$$ and we set
\begin{align*}
w_{j,\ell}&=j\quad \text{ for }(j,\ell)\in\E,\\
\lambda_{j,\ell}&= \sqrt{j(j+1)+m}\text{ for }(j,\ell)\in\L,\\
(\om_0)_{j,\ell}(\r)&=\sqrt{j(j+1)+m+\de\r_j}\text{ for }(j,\ell)\in\A,\\
\zeta&=(\xi_a,\eta_a)_{a\in\L}.
\end{align*}
With this notation the Klein Gordon Hamiltonian $H$ reads (up to a constant)
$$H(r,\theta,\zeta)= \lan\om_0(\r),r\ran+\sum_{a\in\L} \la\xi_a\eta_a+ \eps f(r,\theta,\zeta)$$
where 
$$f(r,\theta,\zeta)=
\int_{\S^2}G\left(x, \hat u(r,\theta,\zeta)(x)\right)dx .$$

\begin{lemma}\label{AOK2}
Hypothesis A1, A2 and A3 hold true with $\D=[1,2]^n$ and  
\be\label{d0}\de_0=\Big(\frac{\de}{4\max({w_a,\ a\in\A})}\Big)^3.\ee \end{lemma}
\proof
Hypothesis A1 is clearly satisfied with $c_0=1/2$ and $\ga=1$.\\
On the other hand choosing $z\equiv z_k=\frac k{|k|}$ we have
\be\label{ouf} \nazz  (\langle k,\omega\rangle) \geq \frac{\de}{2\max({w_a,\ a\in\A})} |k|\quad \text{ for all }k\neq 0\ee
while 
\be\label{ouf2} \nazz  \la =0\quad \text{ for all } a\in\L .\ee
Then for all $k\neq 0$ the second part of the alternatives (i)--(iv) in Hypothesis A2 are satisfied choosing $$\de_0\leq \de_* :=\frac{\de}{4\max({w_a,\ a\in\A})}.$$
It remains to verify A3. Without loss of generality we can assume $w_a\leq w_b.$\\
First  denoting
$$F_\ka(k,a,b):=\{\r\in\D\mid |\lan \om,k\ran +\la-\lb|\leq \ka\},$$
we have using \eqref{ouf} that
$$\meas F_\ka(k,a,b)\leq C \frac{\ka}{\delta_*}.$$
On the other hand, defining
$$G_\nu(k,e):=\{\r\in\D\mid |\lan \om,k\ran+e|\leq 2\nu\},$$
we have, using again \eqref{ouf} that
$$\meas G_\nu(k,e)\leq C \frac{\nu}{\delta_*}.$$
Further $|\lan \om,k\ran+e|\leq 1$ can occur only if $|e|\leq C|k|$ and thus
$$G_\nu=\bigcup_{\substack{{0<|k|\leq N}\\{e\in\Z}}} G_\nu(k,e)$$
has a Lebesgue measure less than $CN^{n+1} \frac{\nu}{\delta_*}.$\\
Now we remark that 
$$|j+\frac 14-\sqrt{j(j+1)+m}|\leq \frac {m+1}{2j}$$
from which we deduce
$$|\la-\lb -(w_a-w_b)|\leq \frac {m+1}{w_a}.$$
Therefore for $\r\in \D\setminus G_\nu$ and $w_a\geq \frac 2 \nu$ we have for all  $0<|k|\leq N$ 
$$|\lan \om,k\ran +\la-\lb|\geq \nu.$$
Finally  $w_a\leq \frac 2 \nu$ and $|\lan \om,k\ran +\la-\lb|\leq 1$ leads to $w_b\leq \frac 2 \nu+CN$ and thus,
if we restrict $\r$ to 
$$\D'=\D\setminus \Big[G_\nu \cup \Big(  \bigcup_{\substack{{0<|k|\leq N}\\{w_a,w_b\leq \frac 2 \nu+CN}}}F_\ka(k,a,b) \Big)\Big]$$
we get
$$|\lan \om,k\ran +\la-\lb|\geq \min(\ka,\nu), \quad 0<|k|\leq N,\ a,b\in\L.$$
Further
$$\meas \D\setminus \D'\leq CN^{n+1} \frac{\nu}{\delta_*}+(\frac 2 \nu+CN)^2N^n\frac{\ka}{\delta_*}.$$
Then choosing $\nu= \ka ^{1/3}$ and $\de_0=\de_*^{3}$, this measure is controlled by
$$CN^{n+2}\big(\frac \ka {\de_0}\big)^{1/3}$$
and we have
$$|\lan \om,k\ran +\la-\lb|\geq\ka, \quad \text{for } \r\in\D',\  0<|k|\leq N \text{ and } a,b\in\L.$$
Now we remark that for $|\la-\lb|\geq 2|\lan \om,k\ran|$, 
$$|\lan \om,k\ran +\la-\lb|\geq \frac 12 |\la-\lb|\geq \frac 14 (1+|w_a-w_b|)\geq \ka (1+|w_a-w_b|)$$
if we assume $\ka\leq \frac 14$. \\
On the other hand, when $|\la-\lb|\leq 2|\lan \om,k\ran|\leq CN$, 
$$|\lan \om,k\ran +\la-\lb|\geq  \tilde\ka(1+|w_a-w_b|)$$
where $\tilde\ka=\frac\ka {1+CN}$. Thus we get 
$$|\lan \om,k\ran +\la-\lb|\geq\tilde\ka(1+|w_a-w_b|), \quad \text{for } \r\in\D',\  0<|k|\leq N \text{ and } a,b\in\L$$
with
$$\meas  \D\setminus \D'\leq CN^{n+3}\big(\frac {\tilde\ka} {\de_0}\big)^{1/3}.$$
  \endproof
\begin{remark}\label{rem-m=0}
When $m=0$, Hypothesis A3 is verified for $ \r\in \D'=\D\setminus \G_\ga$ and thus we can choose $\de_0=\de_*$. 
\end{remark}  
\begin{lemma}\label{FOK2}
Assume that $(x,u)\mapsto g(x,u) $ is real analytic on $\S^2\times\R$ and $s>1$ then there exist $\s>0$, $\mu>0$ such that
$$\O^s(\s,\mu)\times\D\ni(r,\theta,\zeta;\r)\mapsto f(r,\theta,\zeta;\r)$$
belongs to $\Tc^{s,1/4}(\D,\s,\mu)$.
\end{lemma}
\proof First we notice that $f$ does not depend on the parameter $\r$.
Due to  the analyticity of $g$ and the fact that\footnote{$s>1$ is needed to insure that $Y_s$ is an algebra.} $s>1$, there exist positive
 $\s$ and $\mu$  such that $f:\O(\s,\mu)\times\D\to \C$ is  a $C^1$-function, analytic in the first variables $(r,\theta,\zeta)$, whose gradient in $\zeta$ analytically 
maps $Y_s$ to itself (e.g., see in  \cite{EK10}). \\
It remains to verify that $\nabla^2_\zeta f(r,\theta,\zeta;\r)\in\M_{1/4}$.\\
We have
\be
\frac{\partial^2 f}{\partial\xi_a\xi_b}=\frac{\partial^2 f}{\partial\eta_a\eta_b}=\frac{\partial^2 f}{\partial\xi_a\eta_b}=\frac { 1}{2\la^{1/2}\lb^{1/2}}\int_{\S^2}\partial_{u}g(x,\hat u(x))\Psi_a\Psi_b \, dx
\ee
where $\hat u(x)\equiv \hat u(r,\theta,\zeta)(x)$ is given by \eqref{uhat}. We note that for $s>1$ and $(r,\theta,\zeta)\in\O^s(\s,\mu)$, $x\mapsto \hat u(x)$ is bounded on $\S^2$.\\
It remains to prove that  the infinite matrix $M$ defined by
$$M_a^b= \frac { 1}{\la^{1/2}\lb^{1/2}}\int_{\S^2}\partial_{u}g(x,\hat u)\Psi_a\Psi_b \, dx$$
belongs to $\M_{1/4}$, i.e.
$$\sup_{a,b\in\L}w_a^{1/4}w_b^{1/4}\left\| M_{[a]}^{[b]}\right\|_{HS}<\infty.$$
Let us denote $\Pi_b$ the orthogonal projection in $L^2(\S^2)$ on the eigenspace $E_b:=\text{Span}\{\Psi_{d}\mid w_d=w_b\}$.We have
\begin{align*}
\left\| M_{[a]}^{[b]}\right\|_{HS}^2&= \sum_{c\in[a],d\in[b]}\frac { 1}{\la\lb}\left|\int_{\S^2}\partial_{u}g(x,\hat u)\Psi_c\Psi_d \, dx\right|^2\\
&= \frac { 1}{\la\lb}\sum_{c\in[a]}\left\|\Pi_{b}(\partial_{u}g(x,\hat u)\Psi_c)   \right\|_{L^2(\S^2)}^2\\
&\leq \frac { 1}{\la\lb}\int_{\S^2}|\partial_{u}g(x,\hat u)|^2\Big(\sum_{c\in[a]}|\Psi_c|^2\Big)dx\\
&\leq \frac { C}{\lb}\int_{\S^2}|\partial_{u}g(x,\hat u)|^2dx
\end{align*}
where we used the Unsšl\"od's theorem\footnote{This result is a consequence of the two following facts:
\begin{itemize}
\item $T(x)=\sum_{c\in[a]}|\Psi_c(x)|^2$ is the trace of the integral-kernel of the projection operator on $E_{[a]}$.
\item This projection commutes with the rotations on $\S^2$ and thus $T$ does not depend on $x$.
\end{itemize}}:
$$\sum_{c\in[a]}|\Psi_c(x)|^2=\frac{\card E_{[a]}}{4\pi}\leq  C\la,\quad x\in\S^2$$
where $C$ is an universal constant.
Similarly we have
$$\left\| M_{[a]}^{[b]}\right\|_{HS}^2\leq  \frac { C}{\la}\int_{\S^2}|\partial_{u}g(x,\hat u)|^2dx$$
and thus for all $a,b\in\E$
$$w_a^{1/4}w_b^{1/4}\left\| M_{[a]}^{[b]}\right\|_{HS}\leq  \ C\Big(\int_{\S^2}|\partial_{u}g(x,\hat u)|^2dx\Big)^{\frac 1 2}\leq C'$$
for a constant $C'$ depending only on $g$.
\endproof

So  Main Theorem applies (for any choice of the vector $I\in[1,2]^\A$) and Theorem \ref{thmKG} is proved.
 
\begin{remark} We can also consider the Klein Gordon equation \eqref{KG} on the higher dimensional sphere $\S^3$ but in this case the same proof as in Lemma \ref{FOK2} will lead to $f\in\Tc^{s,0}(\D,\s,\mu)$ (since then $\card E_{[(j,\ell)]}=j^2 $). So in order to apply our KAM theorem we would need to consider the regularized Klein Gordon equation ($\b>0$):
\be
\label{KGR} 
(\partial_{t}^2-\Delta+m^2)u=\Lambda^{-\b}\partial_{2}g(x,\Lambda^{-\b}u),\quad t\in\R,\ x\in\S^3 .
\ee

\end{remark}
\begin{remark}\label{rem-beam} We can also consider the Beam equation on the torus $\T^d$  with convolution potential in a Sobolev like phase space\footnote{The same equation but in an analytic phase space were considered in \cite{EGK1,EGK2}.}:
\be \label{beam}u_{tt}+\Delta^2 u+m u+V\star u + \eps \partial_u G(x,u)=0 ,\quad   x\in \T^d.
\ee
 Here $m$ is the mass,  $G$ is a real analytic function on $\T^d\times \R$ and at least of order 3 at the origin. The convolution potential $V:\ \T^d\to \R$ is supposed to be analytic with
   real  positive Fourier coefficients $\hat V(a)$, $a\in\Z^d$. 
Actually following \cite{EGK1} and the proof of Lemma \ref{FOK2}, it remains to control the HS-norm of the infinite matrix\footnote{Here $\la=\sqrt{|a|^4+m}$ and $\Psi_a(x)=e^{ia\cdot x}$, $a\in\Z^d$.}
$$M_a^b= \frac { 1}{\la^{1/2}\lb^{1/2}}\int_{\T^d}\partial^2_uG(x,u)\Psi_a\Psi_b \, dx$$
restricted to the block defined by $[a]=\{b\in\Z^d\mid |a|=|b|\}$.  We have
$$\left\| M_{[a]}^{[b]}\right\|_{HS}^2\leq \frac { C}{\la\lb}\sup_{x\in\T^d}\sum_{c\in[a]}|e^{ic\cdot x}|^2\leq \frac { C}{\la\lb}|a|^{d-1}\leq C\frac{\la^{\frac{d-3}{2}}}{\lb}$$
which leads by symmetrization to
$$\left\| M_{[a]}^{[b]}\right\|_{HS}^2\leq\frac C {(\la\lb)^{\frac{5-d}{4}}}$$
and then $M\in \M_{\frac{5-d}{8}}$. Thus we can apply our theorem as soon as  $d\leq 4$.
\end{remark}

\subsection{The regularized quantum harmonic operator in $\R^2$.}
Let   $$T=-\Delta + |x|^2=-\Delta +x_1^2+x_2^2$$
 be the 2-dimensional quantum harmonic oscillator. 
Its spectrum is  the sum of $2$-copies of the odd integers set, i.e. the spectrum of $T$ equals $2\N=\{2,4,\cdots\}$.\\
For $2j\in 2\N$ we denote the associated eigenspace $E_j$ whose dimension is
$$\sharp\{ (i_1, i_2)\in(2\N-1)^2 \mid i_1+i_2=2j \}=j.$$
We denote $\{\Phi_{j,l}$,  $l=1,\cdots,j\}$, the basis of $E_j$ obtained by  $2$-tensor product  of Hermite functions:   $\Phi_{j,l}=\phi_{i_1} \otimes \phi_{i_2}$ with $i_1=2\ell -1$ and $i_2=2j-(2\ell-1) $.
Then  setting
$$\Nd:=\{(j,\ell)\in\N\times\N\mid \ell=1,\cdots,j\}$$
$(\Phi_a)_{a\in\Nd}$ is a basis of  $L^2(\R^2)$. \\
The Hermite multiplier M is defined on this basis  by
\be\label{M}
M \Phi_a=\r_a\Phi_a\quad \text{ for } a\in\Nd
\ee
where $(\r_a)_{a\in\Nd}$ is a bounded sequence of real numbers.

\medskip

In this subsection we consider the following nonlinear Schr\"odinger equation in $\R^2$
\be\label{harmo}
i\, u_t=Tu +M\cdot u+ T^{-\b}\partial_{2}F(x,T^{-\b}u,T^{-\b}\bar u), \quad t\in\R,\ x\in\R^2
\ee
where $M$ is a Hermite multiplier defined below and $F$ is a smooth function.\\
We focus on two choices of non linearity:
\begin{itemize}
\item {\bf The regularized cubic NLS} which corresponds to the choice
$$F_{NLS}(x,u,\bar u)=\pm\frac 1 4 |u|^4$$
which correspond in the non regularized case ($\b=0)$ to cubic NLS 
$$i\, u_t=-\Delta u +|x|^2u+M\cdot u\pm |u|^2u.$$
\item {\bf The regularized Hartree  equation} which corresponds to the choice
$$F_{Hartree}(x,u,\bar u)=\int_{\R^2} |u(x)|^2|u(y)|^2\phi(x-y)\, dy$$
where $\phi$ is a smooth function. This case correspond in the non regularized case $(\b=0)$ to the Hartree equation
$$i\, u_t=-\Delta u +|x|^2u+M\cdot u+(\phi\star |u|^2)u.$$
\end{itemize}
Let 
\begin{align}\begin{split} \label{H}
\tilde{H^s}=\{&f\in H^s(\R^2,\C) | x\mapsto {x}^{\alpha}\partial^\b f \in L^2(\R^2)\\
&\mbox{ for any } \alpha,\  \beta\in \N^2 \mbox{ satisfying } 0\leq |\alpha|+|\b|\leq s \}
\end{split}\end{align}
where $H^s(\R^2,\C)$ is the standard Sobolev space on $\R^2$. We note that, for any $s\geq 0$,
the domain of $T^{s/2}$ is $\tilde{H^s}$ (see for instance \cite{Helf84} Proposition 1.6.6) and that for $s>1$, $\tilde{H^s}$ is an algebra.\\
In the phase space $\tilde{H^s}\times\tilde{H^s}$ endowed with the symplectic 2-form $idu\wedge d\bar u$ equation \eqref{harmo} reads as the Hamiltonian system associated to the Hamiltonian function
\begin{align}\label{Hharmo}
H(u,\bar u)&= \int_{\R^2} \big(|\nabla u|^2+|x|^2|u|^2+F(x,T^{-\b}u,T^{-\b}\bar u)\big)dx\\
\nonumber &=H_0(u,\bar u)+P(u,\bar u).
\end{align}
In particular, for the regularized cubic NLS equation, the perturbation term reads
$$P_{NLS}=\pm \frac 1 4 \int_{\R^2}|T^{-\b}u|^4 dx$$
while for the regularized Hartree equation we have
$$P_{Hartree}=\int_{\R^2}\int_{\R^2} |T^{-\b} u(x)|^2|T^{-\b} u(y)|^2\phi(x-y)\, dxdy.$$
Decomposing $u$ and $\bar u$ on the basis $(\Phi_{j,l})_{(j,l)\in\Nd}$,
$$u=\sum_{a\in\Nd}\xi_{a}\Phi_a,\quad \bar u=\sum_{a\in\Nd}\eta_{a}\Phi_a$$
the phase space $(u,\bar u)\in\tilde{H^s}\times\tilde{H^s}$ becomes the phase space $(\xi,\eta)\in Y_s$
$$
Y_s=\{\zeta=(\zeta_a\in\C^2,\ a\in \Nd)\mid \|\zeta\|_s<\infty\} 
$$
where
$$
\|\zeta\|_s^2=\sum_{a\in\L}|\zeta_a|^2 w_a^{2s}$$
and
$$w_{j,\ell}=j\quad \text{ for }(j,\ell)\in\Nd.$$
We endowed $Y_s$ with the symplectic structure $id\xi\wedge d\eta$.\\
Then the Hamiltonian reads
$$H(\xi,\eta)= \sum_{a\in\Nd} (w_a+\r_a)\xi_a\eta_a+\int_{\R^2}F\left(x,\sum_{a\in\Nd}\frac{\xi_{a}}{w_a^\beta}\Phi_a, \sum_{a\in\Nd}\frac{\eta_{a}}{w_a^\beta}\Phi_a    \right).$$
Let $\A\subset \Nd$ a finite subset of cardinal $n$ which is {\it admissible} i.e. which satisfies 
$$\A\ni(j_1,\ell_1)\neq(j_2,\ell_2)\in\A \Rightarrow j_1\neq j_2.$$
We fix $r_a^0\in[1,2]$ for $a\in\A$, the initial $n$ actions, and we write the modes $\A$ in action-angle variables:
$$\xi_a=\sqrt{r_a^0+r_a}e^{i\theta_a},\quad \eta_a=\sqrt{r_a^0+r_a}e^{-i\theta_a}.$$
We define $\L=\Nd\setminus\A$ 
and we assume 
$$\r_{j,l}=\r_j \text{ for } (j,\ell)\in \A\ ;\ \r_{j,l}=0 \text{ for } (j,\ell)\in \L.$$
We set
\begin{align*}
w_{j,\ell}&=j\quad \text{ for }(j,\ell)\in\Nd,\\
\lambda_{j,\ell}&=j \text{ for }(j,\ell)\in\L,\\
\om_{j,\ell}&=j+\r_{j}\text{ for }(j,\ell)\in\A.
\end{align*}
With this notation $H$ reads (up to a constant)
$$H(\xi,\eta)= \om\cdot r+\sum_{a\in\L} \la\xi_a\eta_a+\int_{\R^2}F\left(x,\sum_{a\in\Nd}\frac{\xi_{a}}{w_a^\beta}\Phi_a, \sum_{a\in\Nd}\frac{\eta_{a}}{w_a^\beta}\Phi_a    \right).$$
\begin{lemma}\label{AOK}
Hypothesis A1, A2 and A3 hold true with $\delta =1/2$ and $\D=[0,1]^n$.\end{lemma}
\proof,
The asymptotics A1 are verified with $\ga=1$.\\
Next we remark that  for $a,b\in\A$, $(\nabla_\r \om_a)_b=\delta_{a,b}$ where $\delta$ denotes the Kronecker symbol while $\nabla_\r \la=(0,\cdots,0)^t$ .Thus A2 holds true with $\delta =1/2$ and $\D=[0,1]^n$.\\
Finally, noticing that $\la-\lb\in\Z$ we easily deduce A3 from A2 as in the proof of Lemma \ref{AOK}.  
\endproof
\begin{lemma}\label{FOK}
%For both choices $F=F_{NLS}$ and $F=F_{Hartree}$, 
Assume that $(x,z,\bar z)\mapsto F(x,z,\bar z) $ is real analytic in $x$, $\Re z$, $\Im z$ and assume that $s>1$ then there exist $\s>0$, $\mu>0$ such that
\begin{align*}&\O^s(\s,\mu)\ni(r,\theta,\zeta)\mapsto P(r,\theta,\zeta)=\\ 
&\int_{\R^2}F\left(x,\sum_{a\in\A}\frac{\sqrt{r_a^0+r_a}e^{i\theta_a}}{w_a^\beta}\Phi_a+\sum_{a\in\L}\frac{\xi_{a}}{w_a^\beta}\Phi_a, \sum_{a\in\A}\frac{\sqrt{r_a^0+r_a}e^{-i\theta_a}}{w_a^\beta}\Phi_a\sum_{a\in\L}\frac{\eta_{a}}{w_a^\beta}\Phi_a    \right)\end{align*}
belongs to $\Tc^{s,\beta}(\D,\s,\mu)$ for any $\b\geq 0$.
\end{lemma}
\proof
We focus on the case $F=F_{NLS}$ and  on the most restrictive hypothesis: $\nabla^2_\zeta P(x,\r)\in\M_\b$.\\
We have
\begin{align*}
\frac{\partial^2 P}{\partial\xi_a\xi_b}&=\frac {\pm 1}{2w_a^\b w_b^\b}\int_{\R^2}{\bar u}^2\Phi_a\Phi_b \, dx,\\
\frac{\partial^2 P}{\partial\eta_a\eta_b}&=\frac {\pm1}{2w_a^\b w_b^\b}\int_{\R^2}u^2\Phi_a\Phi_b \, dx,\\
\frac{\partial^2 P}{\partial\xi_a\eta_b}&=\frac1{w_a^\b w_b^\b}\int_{\R^2}|u|^2\Phi_a\Phi_b \, dx.
\end{align*}
So it remains to prove that  the infinite matrix $M$ defined by
$$M_a^b= \int_{\R^2}|u|^2\Phi_a\Phi_b \, dx$$
belongs to $\M_0$, i.e.
$$\sup_{a,b\in\L}\left\| M_{[a]}^{[b]}\right\|_{HS}<\infty.$$
Let us denote $\Pi_b$ the orthogonal projection in $L^2(\R^2)$ on the eigenspace $E_b:=\text{Span}\{\Phi_{d}\mid w_d=w_b\}$.We have
\begin{align*}
\left\| M_{[a]}^{[b]}\right\|_{HS}^2&= \sum_{c\in[a],d\in[b]}\left|\int_{\R^2}|u|^2\Phi_c\Phi_d \, dx\right|^2\\
&= \sum_{c\in[a]}\left\|\Pi_{b}(|u|^2\Phi_c)   \right\|_{L^2(\R^2)}^2\leq \sum_{c\in[a]}\left\||u|^2\Phi_c   \right\|_{L^2(\R^2)}^2\\
&\leq \int_{\R^2}|u|^4\Big(\sum_{c\in[a]}|\Phi_c|^2\Big)dx\\
&\leq C\int_{\R^2}|u|^4dx
\end{align*}
where we used the crucial property of the quantum Harmonic oscillator (see Lemma \ref{osci} just below):
\be\label{psiC}\sum_{c\in[a]}|\Phi_c(x)|^2\leq C,\quad x\in\R^2\ee
where $C>0$ does not depend on $a$. 
\endproof
The function $K_{a}(x,y):=\sum_{c\in[a]}\Phi_c(x)\Phi_c(y)$ is the  integral kernel of the projection operator on $E_a$. It does not depend on the choice of the basis of $E_a$. %More generally let us denote by $\Pi(I)$ the spectral projector of $T$ on $I$ a subset of $\R^+$. The range of $\Pi(I)$ is spanned by $\{ \Phi
From \cite{Than,Kara} (see also \cite{KTZ}) we learn the following (non trivial) estimate
\begin{lemma}\label{osci}
Let $d\geq 2$, there exists $C>0$ such that for all $a$
$$K_a(x,x)\leq C\lambda_a^{d/2-1} \quad \text{for all } x\in\R^d.$$
\end{lemma}
In particular, in dimension $d= 2$, we deduce \eqref{psiC}.
\section{Poisson brackets and Hamiltonian flows.}
It turns out that the  space $\Tc^{s,\beta}(\D,\s,\mu)$ is not stable by Poisson brackets.
Therefore, in this section, we first define a new  space $\Tc^{s,\beta+}(\D,\s,\mu)\subset \Tc^{s,\beta}(\D,\s,\mu)$ and then we prove a structural stability  which is essentially contained in the claim
 $$\{\Tc^{s,\beta+}(\D,\s,\mu)\ ,\ \Tc^{s,\beta}(\D,\s,\mu)\} \in \Tc^{s,\beta}(\D,\s,\mu).$$  
We will also study the hamiltonian flow generated by hamiltonian function in $\Tc^{s,\beta+}(\D,\s,\mu)$.\\

\subsection{New Hamiltonian space}
We introduce $\Tc^{s,\beta+}(\D,\s,\mu)\subset \Tc^{s,\beta}(\D,\s,\mu)$defined by
$$\Tc^{s,\beta+}(\D,\s,\mu)=\{f\in \Tc^{s,\beta}(\D,\s,\mu)\mid\partial^j_r\nabla_\zeta f\in L_{\b}^+\ ,\ \partial^j_\r \nabla^2_\zeta f\in\M_\b^+,\ j=0,1\}$$
where 
$$\M_\b^+=\{M\in\M\mid |M|_{\b+} <\infty \},\quad L_{\b}^+=\{\zeta\in L\mid |\zeta|_{\b+} <\infty \} $$
and 
\begin{align*}|M|_{\b+}&=\sup_{a,b\in\L}(1+|w_a-w_b|)w_a^{\b}w_b^{\b}\left\| M_{[a]}^{[b]} \right\|_{HS}\\
|\zeta|_{\b+}&=\sup_{a\in\L}w_a^{\b +1}|\zeta_a|.
\end{align*}
We endow $\Tc^{s,\beta+}(\D,\s,\mu)$ with the norm
$$[f]_{\s,\mu,\D}^{s,\b+}=[f]_{\s,\mu,\D}^{s,\b}+\sup_{j=0,1}\Big(\mu |\partial^j_r\nabla_\zeta f|_{\b+}+\mu^2|\partial^j_\r \nabla^2_\zeta f|_{\b+}\Big).$$

\begin{lemma}\label{product} Let $\b>0$ there exists a constant $C\equiv C(\b)>0$ such that
\begin{itemize}
\item[(i)]
Let $A\in \M_\b^+$ and $B\in \M_\b$ then $AB$ and $BA$ belong to $\M_\b$ and
$$|AB|_{\b},\ |BA|_{\b}\leq C|A|_{\b+}|B|_{\b}.$$
\item[(ii)]
Let $A,B\in \M_\b^+$  then $AB$ and $BA$ belong to $\M_\b^+$ and
$$|AB|_{\b+},\ |BA|_{\b+}\leq C|A|_{\b+}|B|_{\b+}.$$
\item[(iii)] Let $A\in \M_\b^+$ and $\zeta\in Y_s$ for some $s\geq0$ then $A\zeta \in L_\b$ and
$$|A\zeta|_\b\leq C|A|_{\b+}\|\zeta\|_s.$$
\item[(iv)] Let $A\in \M_\b$ and $\zeta\in L_\b^+$ then $A\zeta \in L_\b$ and
$$|A\zeta|_\b\leq C|A|_{\b}|\zeta|_{\b+}.$$
\item[(v)] Let $A\in \M_\b$ and $\zeta\in Y_s$ for some $s\geq 1$ then $A\zeta \in L_\b$ and
$$|A\zeta|_\b\leq C|A|_{\b}\|\zeta\|_s.$$
\item[(vi)] Let $A\in \M_\b+$ and $\zeta\in Y_s$ for some $s\geq 1$ then $A\zeta \in L_{\b+}$ and
$$|A\zeta|_{\b+}\leq C|A|_{\b+}\|\zeta\|_s.$$
\item[(vii)] Let $A\in \M_\b+$ and $\zeta\in L_\b^+$  then $A\zeta \in L_{\b+}$ and
$$|A\zeta|_{\b+}\leq C|A|_{\b+}|\zeta|_{\b+}.$$
\item[(viii)] Let $X\in L_\b $ and $Y\in L_\b$  then $A=X\otimes Y\in\M_\b$ and
$$|A|_{\b}\leq C|X|_{\b}|Y|_\b.$$
\end{itemize}
\end{lemma}
\proof
(i)--Let $a,b\in\L$ and for $k\in\{w_a\mid a\in\L\}:=\Nb \subset \N$ denote by $\bk$ an element of $\L$ satisfying $w_\bk=k$.   We have
\begin{align*}
\left\| (AB)_{[a]}^{[b]} \right\|_{HS}&\leq \sum_{k\in \Nb}\left\| A_{[a]}^{[\bk]} \right\|_{HS}\left\| B_{[\bk]}^{[b]} \right\|_{HS}\\
&\leq \frac{|A|_{\b+}|B|_{\b}}{w_a^{\b}w_b^{\b}}\sum_{k\in \Nb}\frac 1 {k^{2\b}(1+|w_a-k|)}\\
&\leq C \frac{|A|_{\b+}|B|_{\b}}{w_a^{\b}w_b^{\b}}
\end{align*}
where we used that $\sum_{k\in \Nb}\frac 1 {k^{2\b}(1+|w_a-k|)}\leq C$ for a constant $C>0$ depending only on $\b$.\\
(ii)--Similarly
\begin{align*}
\left\| (AB)_{[a]}^{[b]} \right\|_{HS}&\leq \sum_{k\in \Nb}\left\| A_{[a]}^{[\bk]} \right\|_{HS}\left\| B_{[\bk]}^{[b]} \right\|_{HS}\\
&\leq \frac{|A|_{\b+}|B|_{\b+}}{w_a^{\b}w_b^{\b}}\sum_{k\in \Nb}\frac 1 {k^{2\b}(1+|w_a-k|)(1+|w_b-k|)}\\
&\leq C \frac{|A|_{\b+}|B|_{\b+}}{w_a^{\b}w_b^{\b}(1+|w_a-w_b|)}
\end{align*}
where we used
$$\{k\geq 1\}\subset \{ k\geq 1, |w_a-k|\geq \frac 1 3 |w_a-w_b|\}\cup \{ k\geq 1, |w_b-k|\geq \frac 1 3 |w_a-w_b|\}.$$
(iii)--We have for any $a\in\L$
%First we note that
%$$\|\zeta\|_s^2=\sum_{a\in\L}w_a^{2s}|\zeta_a|^2=\sum_{k\in 2\N}k^{2s}\|\zeta_{[\bk]}\|^2$$
%where $$\|\zeta_{[\bk]}\|^2=\sum_{a\in[\bk]}|z_a|^2.$$
%Then
\begin{align*}
|(A\zeta)_{[a]}|&= \left| \sum_{j\in \Nb}A_{[a]}^{[\bj]}\zeta_{[\bj]}  \right|\\
&\leq  \sum_{j\in  \Nb}\left\|A_{[a]}^{[\bj]}\right\|_{HS}|\zeta_{[\bj]}| \\
&\leq \sum_{j \in \Nb}\frac{1}{(1+|w_a-j|)w_a^{\b}j^{\b}}\ |A|_{\b+}\|\zeta\|_s\\
&\leq C w_a^{-\b} |A|_{\b+}\|\zeta\|_s\ .
\end{align*}
(iv)--Similarly
\begin{align*}
|(A\zeta)_{[a]}|&= \left| \sum_{j\in  \Nb}A_{[a]}^{[\bj]}\zeta_{[\bj]}  \right|\\
&\leq  \sum_{j\in  \Nb}j^{-1-\b}\left\|A_{[a]}^{[\bj]}\right\|_{HS}j^{1+\b}|\zeta_{[\bj]}| \\
&\leq \sum_{j \in \Nb}\frac{1}{w_a^{\b}j^{1+2\b}}\ |A|_{\b}|\zeta|_{\b+}\\
&\leq C w_a^{-\b} |A|_{\b}|\zeta|_{\b+}\ .
\end{align*}
(v)--Similarly
\begin{align*}
|(A\zeta)_{[a]}|&= \left| \sum_{j\in  \Nb}A_{[a]}^{[\bj]}\zeta_{[\bj]}  \right|\\
&\leq  \sum_{j\in  \Nb}j^{-s}\left\|A_{[a]}^{[\bj]}\right\|_{HS}j^{s}|\zeta_{[\bj]}| \\
&\leq \sum_{j \in \Nb}\frac{1}{w_a^{\b}j^{s+\b}}\ |A|_{\b}\|\zeta\|_s\\
&\leq C w_a^{-\b} |A|_{\b}\|\zeta\|_s\ .
\end{align*}
(vi)--Similarly
\begin{align*}
|(A\zeta)_{[a]}|&= \left| \sum_{j\in  \Nb}A_{[a]}^{[\bj]}\zeta_{[\bj]}  \right|\\
&\leq  \sum_{j\in  \Nb}j^{-s}\left\|A_{[a]}^{[\bj]}\right\|_{HS}j^{s}|\zeta_{[\bj]}| \\
&\leq \sum_{j \in \Nb}\frac{1}{(1+|w_a-j|)w_a^{\b}j^{s+\b}}\ |A|_{\b+}\|\zeta\|_s\\
&\leq C w_a^{-\b-1} |A|_{\b+}\|\zeta\|_s\ 
\end{align*}
where we used that for $s\geq 1$
$$\sum_{j \in \Nb}\frac{1}{(1+|w_a-j|)j^{s+\b}}\leq \frac C{w_a}.$$
(vii)--Similarly 
\begin{align*}
|(A\zeta)_{[a]}|&= \left| \sum_{j\in  \Nb}A_{[a]}^{[\bj]}\zeta_{[\bj]}  \right|\\
&\leq  \sum_{j\in  \Nb}j^{-1-\b}\left\|A_{[a]}^{[\bj]}\right\|_{HS}j^{\b+1}|\zeta_{[\bj]}| \\
&\leq \sum_{j \in \Nb}\frac{1}{(1+|w_a-j|)w_a^{\b}j^{1+2\b}}\ |A|_{\b+}|\zeta|_{\b+}\\
&\leq C w_a^{-\b-1} |A|_{\b+}|\zeta|_{\b+}\ .
\end{align*}
(viii)--Finally 
\begin{align*}
\left\| A_{[a]}^{[b]} \right\|^2_{HS}&= \sum_{c\in[a],d\in[b]}|X_c|^2|Y_d|^2\\
&\leq |X_{[a]}|^2|Y_{[b]}|^2
\end{align*}
and thus
$$|A|_\b\leq |X|_\b|Y|_\b\ .$$

\endproof

\subsection{ Jets of functions.} \label{ss5.1}

For any function $h\in \Tc^{s}(\s,\mu,\D)$ we define its jet $h^T=h^T(x,\r)$ as the following 
 Taylor polynomial of $h$ at $r=0$ and $\zeta=0$:
\be\begin{split}
\label{jet}
h^T=&h_\theta+\langle h_r, r\rangle+\langle h_\zeta,\zeta\rangle+\frac 1 2 \langle h_{\zeta\zeta}\zeta,\zeta \rangle\\
=&h(\theta,0,\r)+\langle \nabla_rh(\theta,0,\r),r \rangle+\langle \nabla_\zeta h(\theta,0,\r),\zeta\rangle+\frac 1 2 \langle \nabla^2_{\zeta \zeta}h(\theta,0,\r)\zeta,\zeta \rangle
\end{split}\ee
Functions of the form  $h^T$ will be called {\it jet-functions.}\\
Directly from  the definition of the norm 
$[h]^{s,\b }_{\s,\mu,\D}$ we get that 
\begin{align}\begin{split}\label{norm2}
&|h_\theta(\theta,\r)|\leq[h]^{s}_{\s,\mu,\D}, 
\quad |h_r(\theta,\r)|\leq\mu^{-2}[h]^{s}_{\s,\mu,\D},\\
&\|h_\zeta(\theta,\r)\|_s \leq \mu^{-1}[h]^{s}_{\s,\mu,\D},
\quad |h_\zeta(\theta,\r)|_\b \leq \mu^{-1}[h]^{s,\b}_{\s,\mu,\D},\\
&| h_{\zeta \zeta}(\theta,\r)|_{\b}\leq\mu^{-2}[h]^{s,\b}_{\s,\mu,\D},
\end{split}\end{align}
for any $\theta\in\T^n_\s$ and any $\r\in \D$. Moreover, the first derivative with respect to $\r$ will satisfy the same estimates.\\
We also notice that by Cauchy estimates we have that for $x\in\O(\s,\mu')$
\be\label{Cauchy}
\|\nabla^2_\zeta h(x)\|_{\L(Y_s,Y_s)}\leq \frac{\sup_{y\in\O(\s,\mu)}\|\nabla_\zeta h(y)\|_{s}}{\mu-\mu'}.\ee
Thus $h_{\zeta\zeta}$ is a linear continuous operator from $Y_s$ to $Y_s$ and
\be \label{norm3}\|h_{\zeta\zeta}(\theta,\r)\|_{\L(Y_s,Y_s)}\leq \mu^{-2}[h]^{s}_{\s,\mu,\D} \ee
for any $\theta\in\T^n_\s$ and any $\r\in \D$.
\begin{proposition}\label{lemma:jet}
For any 
$h\in \Tc^{s,\b}(\s,\mu,\D)$  we have $h^T\in \Tc^{s,\b}(\s,\mu,\D)$,
$$
[h^T]^{s,\b}_{\s,\mu,\D}\leq C  [h]^{s,\b}_{\s,\mu,\D}\,,$$
and, for any   $0<\mu' < \mu$,
$$
[h-h^T]^{s,\b}_{\s,\mu',\D}
\leq  C \left(\frac{\mu'}{\mu}\right)^3 
 [h]^{s,\b}_{\s,\mu,\D}\,,
$$
where $C$ is an absolute constant. 
\end{proposition}

\proof 
We start with the second statement. Consider first the hessian  $\nabla^2_{\zeta\zeta}(h-h^T)(x)$ 
for $x=(\theta,r, \zeta)\in \O^{s}(\s,\mu')$.
 Let us denote $m=\mu'/\mu$. Then for $z\in\overline  D_1=\{z\in\C: |z|\le1\}$ 
 we have $(\theta, (z/m)^2 r,(z/m)\zeta)\in \O^{s}(\s,\mu)$. Consider the function
\begin{equation*}
\begin{split}
&f: D_1\times \O^{s}(\s,\mu')\to \M_\b\,,\\
& (z,x)\mapsto \nabla^2_{\zeta\zeta} h(\theta, (z/m)^2 r,(z/m)\zeta)=h_0(x)+h_1(x) z+\dots \,.
\end{split}
\end{equation*}
It 
is holomorphic and its norm is bounded by $\mu^{-2}[h]^{s,\b}_{\D,\s,\mu}$. So, by the 
Cauchy estimate, $|h_j (x)|_{\b}\leq \mu^{-2}[h]^{s,\b}_{\D,\s,\mu}$ for  $j=1,2,\dots$ and $x\in \O^{s}(\s,\mu')$.
Since $\nabla^2_{\zeta\zeta}(h-h^T)(x)=h_1(x)m+h_2(x)m^2+\cdots,$ then
$\nabla^2_{\zeta\zeta}(h-h^T)$ is holomorphic in $x\in  \O^{s}(\s,\mu)$, and 
$$
|\nabla^2_{\zeta\zeta}(h-h^T)(x)|_{\b}\leq \mu^{-2}[h]^{s,\b}_{\s,\mu,\D}(m+m^2+\dots)\leq \mu^{-2}[h]^{s,\b}_{\s,\mu,\D}\frac{m}{1-m}.
$$
So $\nabla^2_{\zeta\zeta}(h-h^T)$ satisfies the required estimate with $C=2$, if $\mu'\le \mu/2$. 

 Same argument applies to bound  the norms of $\partial_\r \nabla^2_{\zeta\zeta}(h-h^T)$, 
 $h-h^T$  and $ \nabla_\zeta(h-h^T)$ if $\mu'\le \mu/2$,
  and to prove the analyticity of these mappings.
 \smallskip
 
 Now we turn to the  first statement and  write $h^T$ as $h-(h-h^T)$. This implies that $h^T$, $\nabla_\zeta h^T$ and
 $\nabla^2_{\zeta \zeta} h^T$ are analytic on $\O^{s}(\s, \frac12 \mu)$ and that 
 $$
 [h^T]^{s,\b}_{\s, \frac12\mu,\D}\leq C_1 [h]^{s,\b}_{\s,\mu,\D}\,.
 $$
 Since  $h^T$ is a quadratic polynomial, then the mappings  $h^T$, $\nabla_\zeta h^T$ and
 $\nabla^2_{\zeta \zeta} h^T$ are as well  analytic on $\O^{s}(\s, \ \mu)$, and the
  norm  $ [h^T]^{s,\b}_{\s, \mu,\D}$
 satisfies the same estimate, modulo another constant factor, for any $0<\mu'\le\mu$.

 Finally, the estimate for $[h-h^T]^{s,\b}_{\s,\mu',\D}$ when $\mu/2\le\mu'\le\mu$, 
 with a suitable constant $C$, follows from the estimate for $[h^T]^{s,\b}_{\s,\mu,\D}$
 since $[h-h^T]^{s,\b}_{\s,\mu',\D}\le [h^T]^{s,\b}_{\s,\mu,\D} +
 [h]^{s,\b}_{\s,\mu,\D}\,$.
 \endproof

\subsection{Poisson brackets and flows}
The Poisson brackets of functions 
is defined by
\be\label{poisson}
\{f,g\}=\nabla_rf\cdot\nabla_\theta g- \nabla_\theta f\cdot\nabla_r g+\langle J\nabla_\zeta f,\nabla_\zeta g\rangle\,.
\ee
%We recall that $\(f\)$ denotes averaging of a function $f$ in $\theta\in\T^n$.

\begin{lemma}\label{lemma-poisson} Let $s\geq 1$.
Let $f\in\Tc^{s,\beta+}(\D,\s,\mu)$ and $g\in \Tc^{s,\beta}(\D,\s,\mu)$ be two jet functions then for any $0<\s'<\s$  we have $\{f,g\}\in \Tc^{s,\beta}(\D,\s',\mu)$ and
$$[\{f,g\}]_{\s',\mu,\D}^{s,\b}\leq C(\s-\s')^{-1}\mu^{-2}[f]_{\s,\mu,\D}^{s,\b+}[g]_{\s,\mu,\D}^{s,\b}.$$
\end{lemma}
\proof
Let denotes by $h_1$, $h_2$, $h_3$ the three terms on the right hand side of \eqref{poisson}. Since $\nabla_rf(\theta,r,\zeta,\r)= f_r(\theta,\r)$ and $\nabla_rg(\theta,r,\zeta,\r)= g_r(\theta,\r)$ are independent of $r$ and $\zeta$,  the control of  $h_1$ and $h_2$ is straightforward by Cauchy estimates and \eqref{norm2}. \\
We focus on the third term in formula: $h_3=\langle J\nabla_\zeta f,\nabla_\zeta g\rangle$.
 As $\nabla_\zeta f=f_\zeta +f_{\zeta\zeta}\zeta$ and similar for $\nabla_\zeta g$, we have
$$
h_3= \langle  Jf_\zeta, g_\zeta \rangle-\langle  \zeta,f_{\zeta\zeta} J g_\zeta \rangle+\langle g_{\zeta\zeta} J f_\zeta, \zeta \rangle + \langle  g_{\zeta\zeta}J f_{\zeta\zeta}\zeta,\zeta \rangle.
$$
Using \eqref{norm2}, \eqref{norm3} and $\|\zeta\|_s\leq \mu$,
  we get
$$\ 
|h_3(x,\cdot)|\leq C\mu^{-2}[f]^{s}_{\s,\mu,\D}[g]^{s}_{\s,\mu,\D}\,,$$
for any $x\in\O(\s,\mu)$ and $\r\in\D$.\\ Since
$$
\nabla_\zeta h_3=-f_{\zeta\zeta}Jg_\zeta+g_{\zeta\zeta}J f_\zeta+g_{\zeta\zeta}Jf_{\zeta\zeta} \zeta-f_{\zeta\zeta}Jg_{\zeta\zeta} \zeta,$$
then, using Lemma \ref{product}, we get that for  $x\in \O^{s}(\s,\mu)$ with $s\geq1$ and $\r\in\D$
 $$|\nabla_\zeta h_3(x,\cdot)|_{\b}\leq C\mu^{-3}[f]^{s,\b+}_{\s,\mu,\D}[g]^{s,\b}_{\s,\mu,\D}\, .$$
On the other hand using again \eqref{norm3}, we deduce that for  $x\in \O^{s}(\s,\mu)$ and $\r\in\D$
$$\|\nabla_\zeta h_3(x,\cdot)\|_{s}\leq C\mu^{-3}[f]^{s}_{\s,\mu,\D}[g]^{s}_{\s,\mu,\D}\, .$$
Finally, as $\nabla^2h_3=g_{\zeta\zeta}Jf_{\zeta\zeta}-f_{\zeta\zeta}Jg_{\zeta\zeta} $,
 then, using again Lemma \ref{product} we get that for  $x\in \O^{s}(\s,\mu)$ and $\r\in\D$
   $$|\nabla^2h_3(x,\cdot)|_\b\leq C\mu^{-4}[f]^{s,\b+}_{\s,\mu,\D}[g]^{s,\b}_{\s,\mu,\D}\, .$$
\endproof

\subsection{Hamiltonian flows}

To any $C^1$-function $f$ on a domain
 $\O^s(\s,\mu)\times \D$ 
 we associate the hamiltonian equations 
 \be\label{hameq} \left\{\begin{array}{lll}
\dot r = &\nabla_\theta f(r,\theta,\zeta),\\
\dot \theta= &-\nabla_r f(r,\theta,\zeta),\\
\dot \zeta= &J\nabla_\zeta f(r,\theta,\zeta).
\end{array}\right.
\ee
and  denote by $\Phi^t_f\equiv \Phi^t$, $t\in\R$, the corresponding flow-maps (if they exist). 
Now let  $f\equiv f^T$ be  a jet-function
\be\label{f_eq}
f=f_\theta(\theta;\r)+f_r(\theta;\r)\cdot r+\langle f_\zeta(\theta;\r),\zeta\rangle+\frac 1 2 \langle f_{\zeta\zeta}(\theta;\r)\zeta,\zeta \rangle.
\ee
Then  Hamiltonian equations \eqref{hameq} take  the form\footnote{Here and below we often suppress the argument $\r$.}
\be\label{hameq-jet} \left\{\begin{array}{lll}
\dot r = &-\nabla_\theta f(r,\theta,\zeta),\\
\dot \theta= & f_r(\theta),\\
\dot \zeta= &J\left( f_\zeta (\theta)+f_{\zeta\zeta}(\theta)\zeta\right).
\end{array}\right.
\ee
Denote by $V_f=(V_f^r,V_f^\theta,V_f^\zeta)$ the corresponding vector field. It is analytic on any domain 
$\O^{s}(\sigma-2\eta,\mu-2\nu)=:\O_{2\eta, 2\nu}$, where $0<2\eta<\sigma$, $ 0< 2\nu<\mu$.  The
flow-maps $\Phi^t_f$ of $V_f$ on $\O_{2\eta, 2\nu}$ are analytic till they exist. We will study them  till they 
map  $\O_{2\eta, 2\nu}$  to  $\O_{\eta, \nu}$. \\
Assume that 
\be\label{hyp-f1}  
[f]^{s}_{\s,\mu,\D}\leq \frac 12\nu^2 \eta.
\ee
 Then for $x=(r,\theta,\zeta)\in\O_{2\eta, 2\nu}$ by the Cauchy estimate\footnote{Notice that the distance from $\O^{s}(\sigma-2\eta,\mu-2\nu)$ to $\p \O^{s}(\sigma,\mu)$ in the 
the $r$-direction is $4\nu\mu-4\nu^2>4\nu^2$.} and
\eqref{norm3}  we have 
\begin{equation*}
\begin{split}
|V_f^r|_{\C^n} &\le (2\eta)^{-1}\ff\le\nu^2 ,\\
|V_f^\theta|_{\C^n} &\le (4\nu^2)^{-1} \ff\le \eta,\\
\|V_f^\zeta\|_{s} &\le \big(\mu^{-1}+\mu^{-2}\mu\big)\ff\le\nu.
\end{split}
\end{equation*}
Noting that the distance from $\O_{2\eta,2\nu}$ to $\p \O_{\eta,\nu}$ in the 
the $r$-direction is $2\nu\mu-3\nu^2>\nu^2$, in the $\theta$-direction is $\eta$ and in the $\zeta$-direction is
$\nu$, we see that the flow-maps 
\be\label{flow}
\Phi^t_f:  \O^{s}(\sigma-2\eta,\mu-2\nu)\to \O^{s}(\sigma-\eta,\mu-\nu),\qquad 0\le t\le1,
\ee
are well defined and analytic. 

For  $x \in\O_{2s, 2\nu}$ denote $\Phi^t_f(x)=(r(t),\theta(t),\zeta(t))$. Since $V_f^\theta$ is independent from $r$ and 
$\zeta$, then $\theta(t)=K(\theta;t)$, where $K$ is analytic in both arguments.
As $V_f^\zeta=Jf_\zeta+Jf_{\zeta\zeta}\zeta$, where the non autonomous 
linear operator $Jf_{\zeta\zeta}(\theta(t))$ is bounded in the space $Y^c_{s}$ and both the operator and the curve
$Jf_\zeta(\theta(t))$ analytically depend on $\theta$ (through $\theta(t)=K(\theta;t)$), 
then $\zeta(t)=T(\theta,t)+U(\theta;t)\zeta$, where
$U(\theta;t)$ is a bounded linear operator, both $U$ and $T$ analytic in $\theta$.  Similar since $V_f^\zeta$ is a 
quadratic polynomial in $\zeta$ and an affine function of $r$, then $r(t)=L(\theta,\zeta;t)+S(\theta;t)r$, where $S$
is an $n\times n$ matrix and $L$ is a quadratic polynomial in $\zeta$, both analytic in  $\theta$.

The vectorfield $V_f$ is real for real argument, and so are its flow-maps. Since the vector-field is hamiltonian, 
then the flow-maps are symplectic (e.g., see \cite{Kuk2}). We have proven

\begin{lemma}\label{l.flow} 
Let  $0<2\eta<\s$, $0<2\nu<\mu$ and 
 $f =f^T\in\Tc^s(\s,\mu,\D)$ satisfies 
\eqref{hyp-f1}. Then for   $0\le t\le1$ the
 flow maps  $\Phi^t_f$ of  equation \eqref{hameq-jet} define analytical mappings \eqref{flow}
and define
 symplectomorphisms from $\O^{s}(\s-2\eta,\mu-2\nu)$ to $\O^{s}(\s-\eta,\mu-\nu)$. 
 They have the form 
\be\label{Phi1} 
\Phi_f^t:
\left(\begin{array}{lll}r\\ \theta \\ \zeta \end{array}\right)
\to
\left(\begin{array}{lll}
L(\theta,\zeta;t)+S(  \theta;t)r\\
K( \theta;t)\\
T( \theta;t)+U( \theta;t)\zeta
\end{array}\right),
\ee
where  $L(\theta,\zeta;t)$ is  quadratic in $\zeta$,  while 
$U( \theta;t)$ and $S( \theta;t)$ are  bounded linear operators in corresponding spaces. 
\end{lemma}

Our next result specifies the flow-mappings $\Phi_f^t$ and their representation  \eqref{Phi1} when $f\in \Tbp$:

\begin{lemma}\label{changevar}  
Let  $0<2\eta<\s\leq 1$, $0<2\nu<\mu\leq 1$ and 
 $f =f^T\in\Tc^{s,\b+}(\s,\mu,\D)$ satisfies 
\be\label{hyp-f}[f]^{s,\b+}_{\s,\mu,\D}\leq \frac 12 \nu^2 \eta
\ee  Then:\\
1)   Mapping
$L$ is analytic in $(\theta ,\zeta) \in  \T^{\sigma-2\eta}\times \O_\mu(Y_s)$. 
 Mappings $K,T$ and operators  $S$ and $U$ analytically depend on  $\theta\in \T^{\sigma-2\eta}$; their norms 
 and operator-norms  satisfy
\begin{equation}
\begin{split}\label{4.55}
\|S( \theta;t)\|_{\L(\C^n,\C^n)},  \|{}^tU( \theta;t)\|_{\L(Y_s, Y_s)},    \| U( \theta;t)\|_{\L(Y_s, Y_s)}, 
   | U( \theta;t) |_{\b+}\le2,
\end{split}\end{equation}
while for any component $L^j$ of $L$ and any $(\theta,r,\zeta)\in \O^{s}(\s-2\eta,\mu-2\nu)$  we have 
\begin{equation}
\begin{split}\label{4.56}
\|\nabla_\zeta L^j (\theta,\zeta;t) \|_{s} &\le C \eta^{-1}\mu^{-1} \ff ,\\
|\nabla_\zeta L^j (\theta,\zeta;t) |_{\b+} &\le C \eta^{-1}\mu^{-1} \fbp\\
 |\nabla^2_\zeta L^j (\theta,\zeta;t) |_{\b+} &\le C \eta^{-1}\mu^{-2} \fbp.
\end{split}\end{equation}
2) 
The flow maps $\Phi^t_f$ analytically extend to  mappings \\  %on $\C^n\times\T^n_{\s-2s}\times Y^c_{\bar \ga}$,\\
$
\C^n\times\T^n_{\s-2\eta}\times Y_{s}\ni x^0= (r^0,\theta^0,\zeta^0)\mapsto x(t) \in
 \C^n\times \T^n_{\s}\times Y_{s}$, \\ $ x(t)= (r(t),\theta(t),\zeta(t)) $, 
which satisfy
\begin{align}\begin{split}\label{estim-phi1}
&|r(t) -r^0|\leq  4\eta^{-1}\big(1+\mu^{-2}|r^0|+ \mu^{-2}||\zeta^0||^2_{s}\big)\ff,\\
&|\theta(t)-\theta^0|\leq \mu^{-2}\ff,\\
&\|\zeta(t)-\zeta^0\|_{s} \leq 
\left( \mu^{-2}  \|\zeta^0\|_{s}+1\right)\ff,\\
&|\zeta(t)-\zeta^0|_{\b+} \leq 
\left( \mu^{-2}   \|\zeta^0\|_{s}+1\right)\fbp
\end{split}\end{align}
Moreover, $\r$-derivative  of the mapping $x^0\mapsto x(t)$ 
 satisfies  the same estimates as  the increments $x(t)-x^0$. 
 \end{lemma}

\proof
Consider  the  equation for $\zeta(t)$ in \eqref{hameq-jet}:
 \begin{equation}\label{000}
 \dot \zeta(t)= a(t)+ B(t) \zeta(t),\quad \zeta(0)=\zeta^0 \in \O_{\mu-2\nu}(Y_s),
 \end{equation}
 where $a(t)=Jf_\zeta(\theta(t))$ is an analytic curve $[0,1]\to Y_\ga$ and 
 $B(t)=Jf_{\zeta\zeta}(\theta(t))$ is an analytic  curve $[0,1]\to \M$. 
  Both analytically depend on $\theta^0$. 
 By the hypotheses and using \eqref{Cauchy}
 \be\label{aB} ||a(t)||_s\leq \mu^{-1}\ff  ,% \leq \nu, 
 \quad
 \| B\|_{\L(Y_s, Y_s)}\leq \mu^{-2}\ff\leq \frac 12 \nu\leq \frac 12 .\ee
    
 By re-writing \eqref{000} in  the integral form 
 $\zeta(t)=\zeta^0 +\int_0^t (a(t')+B(t')\zeta(t'))\dd t'$ and iterating this relation, we get that 
\be\label{zeta}
 \zeta(t)=a^\infty(t)+(I+B^\infty(t))\zeta^0,
 \ee
 where
$$
 a^{\infty}(t)=\sum_{k\geq 1}\int_{0}^{t}\int_{0}^{t_{1}}\cdots \int_{0}^{t_{k-1}} \prod_{j=1}^{k-1}B(t_{j})a(t_{k})\text{d}t_{k}  \cdots \text{d}t_{2}\,\text{d}t_{1},
$$
 and
   $$
  B^{\infty}(t)=\sum_{k\geq 1}\int_{0}^{t}\int_{0}^{t_{1}}\cdots \int_{0}^{t_{k-1}} \prod_{j=1}^{k}B(t_{j})\text{d}t_{k}  \cdots \text{d}t_{2}\,\text{d}t_{1}.
$$
Due to   \eqref{hyp-f}, \eqref{aB}, for each $k$ and for $0\le t_k\le\dots t_1\le1$ 
 we have that
$$
\|B(t_1)\dots B(t_k)\|_{\L(Y_s, Y_s)}\le (\frac 12)^{k-1}\mu^{-2} [f]^{\ga}_{\s,\mu,\D} .
$$
By this relation and  \eqref{aB} 
  we get that $a^\infty$ and $B^\infty$ are well defined for $t\in[0,1]$ 
and satisfy 
 \begin{align}
 \begin{split}\label{aB-inf}  
   \|B^\infty(t)\|_{\L(Y_s, Y_s)}&\leq  \mu^{-2}\ff ,
 \\
 \|a^\infty(t)\|_{s}& \leq  \mu^{-3} ( \ff )^2 \leq \ff   .
  \end{split}\end{align}
Again,  the curves  $a^\infty$ and $B^\infty$ analytically depend on $\theta^0$. 
 Inserting \eqref{aB-inf}  in \eqref{zeta}  we get that $\zeta=\zeta(t)$ satisfies 
  \eqref{estim-phi1}${}_3$.   \\
On the other hand  for all $t\in[0,1]$, $B\in \M_\b^+$ and 
$$|B(t)|_{\b+}\leq\mu^{-2}\fbp.$$
Therefore using Lemma \ref{product} we get
 \begin{align}
 \begin{split}\label{aB-b}  
   |B^\infty(t)|_{\b+}&\leq  \mu^{-2}\fbp ,
 \\
 |a^\infty(t)|_{\b+}& \leq  \mu^{-3} ( \fbp )^2  \leq \fbp  .
  \end{split}\end{align}
 Inserting \eqref{aB-b}  in \eqref{zeta}  and using again Lemma \ref{product}  we get that $\zeta=\zeta(t)$ satisfies 
  \eqref{estim-phi1}${}_4$.   
  Since in \eqref{Phi1} $U( \theta;t)=I+B^\infty(t)$, then the estimates on $U$ in \eqref{4.55}  follow from \eqref{aB-inf}. 
  \smallskip

 Now  consider  equation for $r(t)$:
 $$
 \dot r(t)= -\alpha(t)-\Lambda(t) r(t),\quad r(0)=r^{0}\in\O_{(\mu-2\nu)^2}(\C^n) 
 $$
 where $ \Lambda (t)= \nabla_\theta f_r(\theta(t))$ and
 \be \label{alpha}
 \alpha(t)= \nabla_\theta f(\theta(t))+\langle \nabla_\theta f_\zeta(\theta(t)),\zeta(t)\rangle+\frac12 \langle \nabla_\theta f_{\zeta\zeta}(\theta(t))\zeta(t),\zeta(t) \rangle.\ee
 The matrix-curve $\Lambda(t)$ and the vector-curve $\alpha(t)$ analytically depend on $\theta^0\in T^n_{\sigma-2\eta}$. Besides, $\alpha(t)$ analytically depends on $\zeta^0\in Y_s$, 
 while $\Lambda$ is  $\zeta^0$-independent.
 
 By the Cauchy estimate and \eqref{hyp-f}, for any $\theta(t)\in T^n_{\sigma-\eta}$ we have 
 \be\begin{split}
 \label{l-la}
&|\Lambda(t)|_{\L(\C^n,\C^n)}\le \eta^{-1}\mu^{-2}\ff\le \frac 12,\\
&|\alpha(t)|\le 2\eta^{-1}  \ff (1+\mu^{-1}\|\zeta^0\|_{s}+\mu^{-2}\|\zeta^0\|^2_{s})
\end{split}
 \ee
 where for the second estimate we used that $\|\zeta(t)-\zeta^0\|_{s}\leq 1+\|\zeta^0\|_{s}$.\\
 Since $\nabla_{\zeta(t)}\alpha(t)=\nabla_\theta f_\zeta(\theta(t)) + 
 \nabla_\theta f_{\zeta\zeta}(\theta(t))\zeta(t)$ and
 $\nabla_{\zeta_0}={}^tU(\theta;t)\nabla_{\zeta}$,   then using \eqref{4.55} we obtain
 \be\label{l-la1}
 \|\nabla_{\zeta^0}\alpha(t)\|_{s}\le 4\eta^{-1}\mu^{-1}\ff(1+\mu^{-1}\|\zeta^0\|_{s}),
 \ee
 and using Lemma \ref{product}
 \be\label{l-la1b}
 |\nabla_{\zeta^0}\alpha(t)|_{b+}\le 4\eta^{-1}\mu^{-1}\fbp(1+\mu^{-1}\|\zeta^0\|_{s}).
 \ee
 Since 
 $\nabla^2_{\zeta^0}\alpha(t) = {}^tU
 \nabla^2_{\theta(t)} \alpha(t) U={}^tU  \nabla_\theta   f_{\zeta\zeta}(\theta(t))U$,
 then  due to  \eqref{4.55} 
  \be\label{l-la2}
 |\nabla^2_{\zeta^0}\alpha(t)|_{s}\le4
  \eta^{-1}\mu^{-2}
  \ff,
 \ee
 while  due to \eqref{4.55} and Lemma \ref{product}
  \be\label{l-la3}
 |\nabla^2_{\zeta^0}\alpha(t)|_{\b+}\le4
  \eta^{-1}\mu^{-2}\fbp .
 \ee
We proceed as for the $\zeta$-equation to derive 
\be\label{r}
r(t)=-\alpha^\infty(t)+(1-\Lambda^\infty(t))r^0,
 \ee
 where
\be\label{alpha-inf}
\alpha^{\infty}(t)=\sum_{k\geq 1}\int_{0}^{t}\int_{0}^{t_{1}}\cdots \int_{0}^{t_{k-1}} \prod_{j=1}^{k-1}\Lambda(t_{j})\alpha(t_{k})\text{d}t_{k}  \cdots \text{d}t_{2}\,\text{d}t_{1},
\ee
 and
\be\label{Lambda-inf}
  \Lambda^{\infty}(t)=\sum_{k\geq 1}\int_{0}^{t}\int_{0}^{t_{1}}\cdots \int_{0}^{t_{k-1}} \prod_{j=1}^{k}\Lambda(t_{j})\text{d}t_{k}  \cdots \text{d}t_{2}\,\text{d}t_{1}.
\ee
Using \eqref{l-la} we get that 
\begin{align*}
|\Lambda^\infty(t)| _{\L(\C^n\times\C^n)}&\leq  \frac 12,  \\
 |\alpha^\infty(t)|_{\C^n}&\leq  \frac 12\eta^{-1}\big( 1+\mu^{-2} \|\zeta^0\|^2_{s}\big)\ff.
 \end{align*}
  Since in \eqref{Phi1} $S(\theta;1)=I- \Lambda^\infty(t)$, then the first estimate in \eqref{4.55} follows.
 Since $\Lambda^\infty(t)$ in \eqref{r} is $\zeta^0$-independent, then $L(\theta,\zeta;t)=-
 \alpha^\infty(t)$. This is a quadratic in $\zeta^0$ expression, and the estimates \eqref{4.56}
 follow from \eqref{l-la1}--\eqref{l-la3} and in view of the estimate for $\Lambda^\infty$  above. \\
  Finally using the estimates for $\Lambda^\infty$ and $\alpha^\infty$ we get from \eqref{r}
that $r=r(t)$ satisfies \eqref{estim-phi1}${}_1$.  \endproof

Next we study how the flow-maps $\Phi^t_f$ transform  functions from  $ \Tb$. 

%For a vector-function $k$ on the torus $\T^n$ we denote
% $\langle k \rangle=(2\pi)^{-n} \int_{\T^n}k(\theta)\dd \theta.$
 \begin{lemma}\label{composition}
  Let $0<2\eta<\s\leq 1$, $0<2\nu<\mu\leq 1$. Assume that  
 $f =f^T\in\Tbp$ satisfies \eqref{hyp-f}.
Let $h\in \Tb$ and denote for $0\leq t\leq 1$ 
 $$h_t(x;\r)=h(\Phi^t_f(x;\r);\r).$$ Then 
 $h_t\in \Tc^{s,\b}(\s-2\eta,\mu-2\nu,\D)$  and
$$
[h_t]^{s,\b}_{\s-2\eta,\mu-2\nu,\D} 
\leq C \frac{\mu}{\nu}\hhh
$$
where $C$ is an absolute constant.
\end{lemma}
\proof
Let us write the flow-map $\Phi^t_f$ as 
$$x^0=(r^0,\theta^0,\zeta^0)\mapsto x(t)=(r(t),\theta(t),\zeta(t)).$$
By Proposition~\ref{changevar}, $h_t(x^0)$ is analytic in $x^0\in \O(\s-2\eta,\mu-2\nu)$. Clearly $|h_t(x^0,\cdot)|\leq \hhh$ for $x^0\in \O(\s-2s,\mu-2\nu)$ and $\r\in\D$. So it remains to estimate the gradient and  hessian of $h(x^0)$.

1) {\it Estimating the  gradient.} Since $\theta(t)$ does not depend on $\zeta^0$,  we have
$$
\frac{\partial h_t }{\partial \zeta^0}= 
\sum_{k=1}^n \frac{\partial h(x(t))}{\partial r_k}\ \frac{\partial r_k(t)}{\partial\zeta^0}+ \sum_{b\in\L} \frac{\partial h(x(t))}{\partial \zeta_b(t)}\ \frac{\partial \zeta_b(t)}{\partial\zeta^0}=\Sigma_1+\Sigma_2.
$$
 i) Since $x(t)\in \O(\s-\eta,\mu-\nu)$,  
 we get  by the Cauchy estimate that
 $$
\left|  \frac{\partial h(x(t))}{\partial r_k}\right|\leq   \frac 1{3\nu^2}\hh .$$
As $\nabla_{\zeta^0} r_k(t)$ was estimated in \eqref{4.56}, then using \eqref{hyp-f} we get 
\begin{align*}
\|\Sigma_1\|_{s} +|\Sigma_1|_{\b}&\le C \nu^{-2} \hhh \,\eta^{-1}\mu^{-1}  \fb\\
&\le C      \mu^{-1}  \hhh\,.
\end{align*}

ii) Noting that $\Sigma_2(r,\theta,\zeta)={}^tU(\theta;t)\nabla_\zeta h$, we get using \eqref{4.55}:
$$\|\Sigma_2\|_{s}+|\Sigma_2|_{\b} \le 4  \mu^{-1} \hhh.$$
Estimating similarly $\frac{\partial}{\partial \r}\frac{\partial h_t}{\partial \zeta}$ we see that for
 $x\in \O(\s-2\eta,\mu-2\nu)$
$$
\| \partial_\r \nabla_{\zeta^0} h_t \|_{s}+|\partial_\r \nabla_{\zeta^0} h_t|_{\b}\leq C\mu^{-1}\hhh.$$

2) {\it Estimating the hessian.}  Since $\theta(t)$ does not depend on $\zeta^0$ and since 
$\zeta(t)$ is affine in 
 $\zeta^0$, then 
 \begin{equation}
 \begin{split}
 \label{feo}
\frac{\partial^2 h_t}{\partial \zeta^0_a \partial \zeta^0_b}(x)&=
\frac{\partial^2 h(x(t))}{\partial \zeta\partial \zeta}  \frac{\partial \zeta(t)}{\partial \zeta^0_a}
\frac{\partial \zeta(t)}{\partial \zeta^0_b}
+
\frac{\partial^2 h(x(t))}{\partial  r^2}  
 \frac{\partial r(t)} {\partial \zeta^0_a}
\frac{\partial r(t)}{\partial \zeta^0_b} \\  
&+
\frac{\partial^2 h(x(t))}{\partial  r \partial\zeta}  
 \frac{\partial r(t)} {\partial \zeta^0_a}
\frac{\partial \zeta(t)}{\partial \zeta^0_b}
+
\frac{\partial h(x(t))}{\partial  r }  
 \frac{\partial^2 r(t)} {\partial \zeta^0_a  \partial \zeta^0_b}\\
& =: \Delta_1+\Delta_2+\Delta_3+\Delta_4.
\end{split}
\end{equation}

i) We have $|\partial^2 h/\partial \zeta_a\partial\zeta_b|_\b\le C\mu^{-2}\hhh$. 
Using this estimate jointly with \eqref{4.55} and Lemma \ref{product} we see that 
$$
|\Delta_1|_{\b}\le C\mu^{-2} \hhh.
$$

ii) Since for $x^0\in  \O^{s}(\s-2s,\mu-2\nu)$ by \eqref{4.56} we have 
$$
 |\nabla r|_{\b}\le C\eta^{-1}\mu^{-1}\fb\,,
$$
 and since by Cauchy estimate $|d_r^2h|\le C \nu^{-4}  \hhh$, we get using Lemma \ref{product}(viii) and \eqref{hyp-f}
$$
 |\Delta_2|_{\b}\le C   \nu^{-4}\hhh   \eta^{-2}\mu^{-2}(\fb)^2
 \le C\mu^{-2} \hhh\,.
$$

iii) For any $j$ we have by the Cauchy estimate that 
$| \frac{\p}{\p r_j}\nabla_\zeta h|_\b\le C \nu^{-3} \hhh$. Therefore by \eqref{4.55} and Lemma \ref{product}
$$
\Big|\sum_{a'} \frac{\partial^2 h}{\partial r_j \partial \zeta_{a'}} \frac{\partial\zeta_{a'}}{\partial \zeta^0_a}\Big|_{\b}
\le C   \nu^{-3} \hhh\, .
$$
Since 
$$ 
|\nabla_{\zeta^0}r_j|_{\b} \le C\eta^{-1}\mu^{-1}\fb\le C\nu^2\mu^{-1}
$$
by \eqref{4.56}, then using again Lemma \ref{product}(viii)  we find that 
$$
|\Delta_3|^\varkappa_{\ga'}\le C\nu^{-1}\mu^{-1} \hhh\,.
$$

iv) As $|\partial h/\partial r(x(t))|\le \nu^{-2}\hhh$ and 
$$
\left| \frac{\partial^2 r}{\partial\zeta^0_a \partial\zeta^0_b} \right|_{\b}
\le C\eta^{-1}\mu^{-2}\fb
$$
by \eqref{4.56}, then 
$$
|\Delta_4|_{\b}\le C\mu^{-2}   \hhh.
$$

The $\rho$-gradient of the hessian leads to estimates similar to the above. So
the lemma is proven.

\endproof

We summarize the results of this section into a proposition.

\begin{proposition}\label{Summarize}
Let  $0<\s'<\s\leq 1,$ $0<\mu'<\mu\leq 1$.There exists an absolute constant $C\geq 1$ such that
\begin{itemize}
\item[(i)]
if $f =f^T\in\Tb$ and 
\be\label{hyp-f'}  
[f]^{s,\b}_{\s,\mu,\D}\leq\\
 \frac1{2} (\mu-\mu')^2 (\s-\s'),
\ee
 then
 for  all $0\le t\le1$,  the Hamiltonian
 flow map  $\Phi^t_f$ is a $\Cc^1$-map 
$$\O^{s}(\s',\mu')\times\D \to\O^{s}(\s,\mu);$$
real holomorphic and symplectic for any fixed  $\rho\in \D$.
Moreover,
$$||\Phi^t_f(x,\cdot)-x||_{s,\D}\le 
C\bigg( \frac1{\s-\s'}+\frac{1}{\mu^2}\bigg) [f]^{s,\b}_{\s,\mu,\D}$$
for any $x\in \O^{s}(\s',\mu')$. 
 
\item[(ii)] if $f =f^T\in\Tbp$ and 
\be\label{hyp-f''}  
[f]^{s,\b+}_{\s,\mu,\D}\leq\\
 \frac1{2} (\mu-\mu')^2 (\s-\s'),
\ee
 then
 for all $0\leq t\leq 1$ and all $h\in \Tc^{s,\b}(\s,\mu,\D)$,
the function
$h_t(x;\r)=h(\Phi^t_f(x,\r);\r)$ belongs to $\Tc^{s,\b}(\s',\mu',\D)$
and
$$[h_t]^{s,\b}_{\s',\mu',\D} 
\leq C \frac{\mu}{(\mu-\mu')  }
 [h]^{s,\b}_{\s,\mu,\D}. $$
\end{itemize}

\end{proposition}

\begin{proof} Take $\s'=\s-2s$ and $\mu'=\mu-2\nu$ and apply lemmas \ref{changevar} and \ref{composition}.
\end{proof}

\section{Homological equation}\label{shomo}
Let us first recall the KAM strategy. 
Let $h_0$ be the normal form   Hamiltonian given by \eqref{h0}
$$
h_0(r,\zeta,\r)=\langle \om_0(\r), r\rangle +\frac 1 2\langle \zeta, A_0(\r)\zeta\rangle$$
satisfying  Hypotheses~A1-A3. Let
 $f$ be a perturbation and 
 $$f^T=f_\theta+\langle f_r, r\rangle+\langle f_\zeta,\zeta\rangle+\frac 1 2 \langle f_{\zeta\zeta}\zeta,\zeta \rangle$$ 
  be its jet  (see \eqref{jet}).
 If $f^T$ was zero,  then $\{\zeta=r=0\}$ would be an invariant 
  $n$-dimensional torus for the Hamiltonian
$h_0+f$. In general we only know that $f$ is small, say $f=\O (\eps)$, and thus $f^T=\O (\eps)$. In order  to 
decrease  the error term 
 we search for a hamiltonian jet $S=S^T=\O (\eps)$ such that its  time-one flow-map
$\Phi_S=\Phi_S^1$ transforms the Hamiltonian $h_0+f$ to
$$
(h_0+f)\circ \Phi_S=h+f^+,
$$
where $h$ is a new normal form, $\eps$-close to $h_0$, and the new perturbation $f^+$ is such that its jet
is much smaller than $f^T$. More exactly, 
$$
h=h_0+\tilde h,\qquad
\tilde h=c(\r)+\langle \chi(\r),r\rangle+ \frac 1 2 \langle \zeta, B(\r)\zeta\rangle=\O(\eps), 
$$
and 
$\ 
\left(f^+\right)^T=\O (\eps^2).
$

As a consequence of the Hamiltonian structure we have (at least formally) that
$$(h_0+f)\circ \Phi_S= h_0+\{ h_0,S \}+f^T+ \O (\eps^2).$$
So to achieve the goal above 
we should  solve the {\it homological equation}:
\be \label{eq-homo}
\{ h_0,S \}=\tilde h-f^T+\O (\eps^2).
\ee
Repeating iteratively 
the same procedure with $h$ instead of $h_0$  etc., we will be forced to solve the homological equation, not
only for the normal form Hamiltonian \eqref{h0}, but for 
more general normal form Hamiltonians
\eqref{h} with $\om$ close to $\om_0$ and $A$ close to $A_0$ .

\medskip 

In this section we will consider a homological equation \eqref{eq-homo}  with $f$  in $\Tc^{s,\beta}(\D,\s,\mu)$ and we will construct a solution $S$  in  $\Tc^{\ga,\beta+}(\D,\s,\mu)$.

\subsection{Four components of the homological equation}\label{ss6.1}

Let $h$ be a  normal form Hamiltonian \eqref{h},
$$
h(r,\zeta,\r)=\langle \om(\r), r\rangle +\frac 1 2\langle \zeta, A(\r)\zeta\rangle\,,
$$
and let us write a jet-function  $S$ as 
$$
S(\theta,r,\zeta)=S_\theta(\theta)+\langle S_r(\theta), r\rangle+\langle S_\zeta(\theta),\zeta\rangle+
\frac 1 2 \langle S_{\zeta\zeta}(\theta)\zeta,\zeta \rangle.
$$ 
Therefore  the Poisson bracket of $h$ and $S$ 
 equals
\begin{equation*}
\begin{split}
\{h, S\}=
(\nabla_\theta\cdot\omega) S_\theta &+\langle (\nabla_\theta\cdot\omega) S_r,r\rangle +
\langle (\nabla_\theta\cdot\omega) S_\zeta, \zeta\rangle \\
&+
\frac12 \langle (\nabla_\theta\cdot\omega) S_{\zeta\zeta},\zeta\rangle 
-\langle AJ S_\zeta,\zeta\rangle  +\langle  S_{\zeta\zeta}JA\zeta,\zeta\rangle.
\end{split}
\end{equation*}
Accordingly the  homological equation \eqref{eq-homo} with $h_0$ replaced by $h$ 
 decomposes into four linear equations. The first two are
\begin{align}\label{homo1}
\langle \nabla_\theta	 S_\theta,\om\rangle =&-f_\theta+c+\O (\eps^2),\\ \label{homo2}
\langle \nabla_\theta	 S_r, \om\rangle =&-f_r+\chi+\O (\eps^2).
\end{align}
In these equations, we are forced to choose 
$$
c(\r)=\( f_\theta(\cdot,\r)\)  \quad \text{and}\quad \chi(\r) =\( f_r(\cdot,\r)\)$$
(see Notation) to achieve that the space mean-value of the r.h.s. vanishes. 
The other two equations are 
\begin{align}\label{homo3}
 \langle\nabla_\theta	 S_\zeta, \om \rangle  -AJ
  S_\zeta=&- f_\zeta+\O (\eps^2),\\ \label{homo4}
 \langle\nabla_\theta	 S_{\zeta\zeta},\om\rangle   -
 AJS_{\zeta\zeta}  +
 S_{\zeta\zeta}JA=&-
 f_{\zeta \zeta}+B+\O (\eps^2)\,,
\end{align}
where  the operator $B$ will be chosen later. 
The most delicate, involving the small divisors as in \eqref{D33},
 is the last equation.

\subsection{The first two equations}\label{homogene}
We begin with equations 
\eqref{homo1} and \eqref{homo2} which are both of the form 
\be\label{homo0}
\langle\nabla_\theta \phi(\theta,\r),\om(\r)\rangle= \psi(\theta,\r)\ee
with $\(\psi\)=0$. 
Here $\om:\D\to\R^n$  is $\Ca^1$ and verifies 
$$|\om-\om_0|_{\Ca^1(\D)}\le\delta_0.$$
Expanding $\phi$ and $\psi$ in Fourier series,
$$\phi=\sum_{k\in\Z^n\setminus \{0\}}\hat \phi (k) e^{ik\cdot \theta}, \quad \psi=\sum_{k\in\Z^n\setminus \{0\}}\hat \psi (k) e^{ik\cdot \theta},$$
we solve eq.  \eqref{homo0}  by choosing
$$
\hat \phi (k) =\ - \frac{i}{ \langle\om, k\rangle}\hat\psi(k), \quad k\in\Z^n\setminus \{0\}; \qquad \hat\phi(0)=0.
$$
 Using Assumption A2~(i) we have, for each $k\not=0$, either that
$$ |\langle\om(\r), k\rangle|\ge \delta_0\quad \forall \r$$
or that
$$\nazz  (\langle k,\omega(\r)\rangle)  \geq \delta_0\quad \forall \r
$$
for a suitable choice of a unit vector $\zz$.
The second case implies that 
$$
 |\langle\om(\r), k\rangle|\ge \ka\,,
 $$
 where $ \ka\le  \delta_0$,
for all $\r$ outside some open 
set $F_k\equiv F_k(\om)$ of Lebesgue measure $\leq \delta_0^{-1}\ka$.\\
Let
$$\D_1=\D\setminus \bigcup_{0<|k|\le N}F_k.$$
Then the closed set $\D_1$ satisfies 
$$\meas(\D\setminus \D_1)\leq N^n\frac {\ka}{\delta_0}\,,$$
and $ |\langle\om(\r), k\rangle|\ge  \ka$ for all $\r\in\D_1$.
Hence, for $\r\in\D_1$ and all $0<|k|\leq N$ we have 
 $$ |\hat \phi (k)| \le \frac{1}{\ka}|\hat\psi(k)|\,.
 $$
Setting $\phi(\theta,\r)=\sum_{0<|k|\leq N}\hat \phi (k,\r) e^{ik\cdot \theta}$, we get that
\be\label{homo0'}
 \langle\nabla_\theta \phi(\theta,\r), \om(\r) \rangle= \psi(\theta,\r)+R(\theta,\r).
\ee
That is, thus defined $\phi$ is an approximate solution of eq.~\eqref{homo0} with  the disparity 
$R(\theta,\r) =-\sum_{|k|> N}\hat \psi (k,\r) e^{ik\cdot \theta}.$
We obtain by a classical argument  that for
 $(\theta,\r)\in \T^n_{\s'}\times \D_1$, $0<\s'<\s$, and $j=0,1$
 \be
\begin{split}\label{esti}
&| \phi(\theta,\r)|\leq \frac{C}{\ka^{}  {(\s-\s')^{n}}} \sup_{|\Im\theta|<\s}|\psi(\theta,\r)| ,\\
&|\p_\r^j  R(\theta,\r)|\leq \frac{C \,e^ {{-\frac12(\s-\s')N}}}{  {(\s-\s')^{n}}}\sup_{|\Im\theta|<\s} 
|\p_\r^j \psi(\theta,\r)| \,, %\quad\text{for any $ 0<\vartheta<1$},
\end{split}
\ee
where $C$ only depends on $n$. If $\psi$ is a real function, then so are
$\phi$ and $R$. \\
Differentiating the formula for 
$\hat\phi(k)$ in $\r$ we obtain \footnote{ Here and below 
  $\chi_Q(k)$ stands for the characteristic function of a set $Q\subset \Z^n$. }
 $$\p_\r \hat\phi(k)= \chi_{|k|\le N}(k)
 \Big(
 \ -\frac{i}{ \langle\om, k\rangle}\p_\r \hat\psi(k)+
 \ \frac{i}{ \langle\om, k\rangle^2} \langle \p_\r\om, k\rangle\hat\psi(k)\Big)\,.
 $$
 From this we derive that
\begin{align*}
|\p_\r \phi(\theta,\r)|\leq &\frac{C'}{\ka^{2}  { (\s-\s')^{n}}}\big( \sup_{|\Im\theta|<\s}|\psi(\theta,\r)|
+\sup_{|\Im\theta|<\s}|\p_\r \psi(\theta,\r)|\big)     \,   ,
\end{align*}
where $C'$ only depends on the derivative of $\om$,  { which is bounded by 
$|\omega_0(\rho)|_{C^1}+\delta_0\le |\omega_0(\rho)|_{C^1}+1$.}

Applying this construction to  \eqref{homo1} and \eqref{homo2}  we get

\begin{proposition}\label{prop:homo12}
Let $\om:\D\to\R^n$  be  $\Ca^1$ and verifying 
${|\om-\om_0|_{\Ca^1(\D)}\le\delta_0.}$
Let $f\in \Tc^\ga(\s,\mu,\D)$ and let $\delta_0\geq \ka>0$, $N\ge1$.
Then there exists  a closed  set $\D_1= \D_1(\om,\ka,N)\subset \D$,  satisfying
$$\meas (\D\setminus \D_1)\leq  C N^{n} \frac\ka{\de_0}, $$
and
\begin{itemize}
\item[(i)] there exist real $\Ca^1$-functions $S_\theta$ and $R_\theta$ on $\T^n_{\s}\times \D_1\to\C$,  
analytic in $\theta$,
such that
 $$\langle \nabla_\theta	 S_\theta(\theta,\r),\om(\r)\rangle  =-f_\theta(\theta,\r)+\(f_\theta(\cdot,\r)\)+R_\theta(\theta,\r)$$
and for all $(\theta,\r)\in \T^n_{\s'}\times \D_1$, $\s'<\s$, and $j=0,1$
 \begin{align*}
 |\p_\r^jS_\theta(\theta,\r)|\leq &\frac{C}{\ka^2 
  {(\s-\s')^{n}}}[f]^s_{\s,\mu,\D_1},\\
 |\p_\r^j R_\theta(\theta,\r)|\leq & \frac{C  { e^{- \frac12(\s-\s')N}}}  {  {(\s-\s')^{n}}}  [f]^s_{\s,\mu,\D_1}\,.
 %\frac{C}{  {\color{red} (\s-\s')^{n}}}[f]^\ga_{\s,\mu,\D_1} {\color{red} e^{- \frac12(\s-\s')N}};
\end{align*}

\item[(ii)] there exist  real $\Ca^1$ vector-functions $S_r$ and  $R_r$  on $\T^n_{\s}\times \D_1$,  
analytic in $\theta$, such that
 $$\langle \nabla_\theta	 S_r(\theta,\r),\om(\r)\rangle =-f_r(\theta,\r)+\( f_r(\cdot;\r)\)+R_r(\theta,\r),$$
 and for all $(\theta,\r)\in \T^n_{\s'}\times \D_1$, $\s'<\s$, and $j=0,1$
 \begin{align*}
 |\p_\r^jS_r(\theta,\r)|\leq &
 \frac{C}{\ka^2
  { (\s-\s')^{n}}}[f]^s_{\s,\mu,\D_1},\\
 |\p_\r^j R_r(\theta,\r)|\leq & \, \frac{C  { e^{- \frac12(\s-\s')N}}}  {  {(\s-\s')^{n}}}  [f]^s_{\s,\mu,\D_1}
 \,.
\end{align*}
\end{itemize}
The  constant $C$ only depends on $ |\om_0 |_{\Ca^1(\D)}$.
\end{proposition}

%Estimates of derivatives of $S$ and $R$ with respect to $\theta$ are obtained by the Cauchy estimate.

\subsection{The third equation}\label{s5.3} 
To begin with, we  recall a result proved
%\footnote{Actually Lemma \ref{EK} is a slight generalization of the result  in \cite{EK10}.....{\color{blue} and the proof is the same?} } 
in the appendix of \cite{EK10}.

\begin{lemma}\label{EK}
Let $A(t)$ be a real diagonal $N\times N$-matrix with diagonal components $a_j$ which are $\Ca^1$ on $I=]-1,1[$,
 satisfying for all $j=1,\dots,N$ and all $t\in I$
$$a'_j(t)\geq\delta_0
 %\quad \text{ or } \quad |a_j(t)|\geq \delta_0
 .$$ Let $B(t)$ be a Hermitian $N\times N$-matrix of class $\Ca^1$ on $I$ such that \footnote{Here 
 $\|\cdot\|$ means the operator-norm of a matrix associated to the euclidean norm on $\C^N$.}
$$\|B'(t)\|\leq \delta_0/2 
%\quad \text{ and } \quad \|B(t)\|\leq \delta_0/2 
,$$ 
for all $t\in I$. 
Then
$$\|(A(t)+B(t))^{-1}\|\leq \frac 1 \eps $$
outside a set of $t\in I$ of Lebesgue measure $\leq C N\eps\delta_0^{-1}$,
where 
$C$ is a numerical constant.
\end{lemma}

Concerning the third component  \eqref{homo3} of the homological equation we have

\begin{proposition}\label{prop:homo3}
Let $\om:\D\to\R^n$  be  $\Ca^1$ and verifying 
${|\om-\om_0|_{\Ca^1(\D)}\le\delta_0.}$  Let $\D\ni\r\mapsto A(\r)\in \NF\cap\M_0$ be $\Ca^1$ and verifying 
\be\label{ass1}
 \| \p_\r^j (A(\r)-A_0(\r))_{[a]} \|_{HS}\le \frac{1}{2}\de_0
 \ee
for $j=0,1$, $a\in\L$ and  $\r\in \D$.
Let $f\in \Tc^s(\s,\mu,\D)$, $0<\ka\leq \frac{\delta_0}2$ and  $N\ge 1$. \\
Then there exists  a closed set $\D_2=\D_2(\om,A,\ka,N)\subset \D$,  satisfying
 $$
 \meas (\D\setminus \D_2)\leq  
 C N^{\exp} 
 \frac{\ka}{\de_0}, $$
and there exist real $\Ca^1$-functions $S_\zeta$ and $R_\zeta$ from  $\T^n  \times \D_2$ to $Y_s$,  
analytic in $\theta$, such that
\be\label{homo3ter}\langle  \nabla_\theta	 S_\zeta(\theta,\r), \om(\r)\rangle  { -A(\r) J}S_\zeta(\theta,\r)=
-f_\zeta(\theta,\r)+R_\zeta(\theta,\r)
\ee
and for all $(\theta,\r)\in \T^n_{\s'}\times \D_2$, $\s'<\s$, and $j=0,1$
 \begin{align*}
\mu\| \p_\r^j S_\zeta(\theta,\r)\|_{s+1}+\mu|\p_\r^j S_\zeta(\theta,\r)|_{\b+}\leq &\frac{C
}{\ka^2 (\s-\s')^{n}}[f]^{s,\b}_{\s,\mu,\D},\\
\mu \|  \p_\r^j R_\zeta(\theta,\r)\|_{s}+\mu|\p_\r^j R_\zeta(\theta,\r)|_{\b}\leq &\frac{C e^{-\frac12(\s-\s')N}}{(\s-\s')^{n}}[f]^{s,\b}_{\s,\mu,\D}
\end{align*}
for $j=0,1$.\\
The exponent $\exp$ only depends on $d, n,\ga$ while the constant $C$ depends on $ |\om_0 |_{\Ca^1(\D)}$,
$ \| A_0 \|_{\Ca^1(\D)}$. 
\end{proposition}

\proof

It is more convenient to deal with the hamiltonian operator $JA$ than with operator $AJ$. Therefore we
multiply eq.~\eqref{homo3ter} by $J$ and obtain for $JS_\zeta$ the equation
\be\label{homo31}
\big\langle  \nabla_\theta	(J S_\zeta)(\theta,\r),\, \om(\r) \big\rangle   -JA(\r) (J S_\zeta)(\theta,\r)=
- Jf_\zeta(\theta,\r)+ JR_\zeta(\theta,\r)
\ee

Let us re-write \eqref{homo31} in the complex variables ${}^t(\xi,\eta)$. 
 For $a\in\L$ {
\be\label{U}
\zeta_a=\left( \begin{array}{ll}p_a\\ q_a\end{array}\right)=U_a\left( \begin{array}{cc}\xi_a\\ \eta_a\end{array}\right), \quad 
U_a=\frac 1 {\sqrt 2}\left( \begin{array}{cc}  1 &  1 \\ -i & i\end{array}\right)\,.
\ee
The symplectic operator $U_a$ transforms the quadratic form 
$(\lambda_a/2) \langle\zeta_a, \zeta_a\rangle$ to $i\lambda_a\xi_a\eta_a$. Therefore, if we denote by $U$
the direct product of the operators diag$\, (U_a, a\in\L)$  then
it transforms $(1/2)\langle\zeta, A_0(\rho)\zeta\rangle\ $ to $\ \sum_{a\in\L}i\lambda_a \xi_a\eta_a$. 
So  it transforms the hamiltonian matrix 
$JA_0(\rho)$
to the diagonal hamiltonian 
matrix
$$
\text{diag}\, \{ i \lambda_a \left( \begin{array}{cc}  -1 &  0 \\ 0 & 1\end{array}\right), a\in\L\}.
$$
Then we make { in \eqref{homo31} the 
substitution   $ JS_\zeta=U S$, $JR_\zeta=U R$ 
and $- Jf_\zeta=U F_\zeta$, where $S={}^t(S^\xi, S^\eta)$, etc. }
%For $a\in\L$ we write the complex 2-vector $S_z(\theta)_a$ as 
%$$ S_z(\theta)_a={}^t(S_a^\xi(\theta),S_a^\eta(\theta))$$ 
%and denote
 %$$S^\xi(\theta)=(S_a^\xi(\theta))_{a\in\L}, \quad S^\eta(\theta)=(S_a^\eta(\theta))_{a\in\L}.$$
%Similarly we introduce $R^\xi(\theta)$ and $R^\eta(\theta)$, $F^\xi(\theta)$ and $F^\eta(\theta)$.\\
In this notation eq.~\eqref{homo3ter}  decouples into two equations 
\begin{align}\begin{split}\label{homo3bis}
&\langle \nabla_\theta  S^\xi, \om\rangle - i\ {}^tQS^\xi= F^\xi+R^\xi,\\
&\langle\nabla_\theta  S^\eta,\om \rangle+  iQS^\eta= F^\eta+R^\eta.
\end{split}\end{align}
Here $Q:\L\times\L\to \C$ is the scalar valued matrix associated to $A$ via the formula \eqref{Q}, i.e. $$Q=\diag \{ \la,\ a\in\L\}+B,
$$
where  $B$ is Hermitian and  block diagonal.

Written in the  Fourier variables,  eq.~\eqref{homo3bis} becomes 
\begin{align}\begin{split}\label{homo3-4}
i(\langle k, \om\rangle + \ {}^tQ)\ \hat S^\xi(k)&= \hat F^\xi(k)+\hat R^\xi(k),\quad k\in \Z^n,\\
i(\langle k, \om\rangle - Q)\ \hat S^\eta(k)&= \hat F^\eta(k)+\hat R^\eta(k),\quad k\in \Z^n.
\end{split}\end{align}
The two equations in \eqref{homo3-4} are similar, so let us consider (for example) the second one, and
let us  decompose it into  its ``components'' over the  blocks $[a]$:
\be\label{homo3-4bis}
i( \langle k, \om(\r) \rangle  + Q(\r)_{[a]}) \hat S_{[a]}(k)=\hat F_{[a]}(k,\r)+\hat R_{[a]}(k)\ee
where the matrix $Q_{[a]}$ is  the restriction of $Q$  to $[a]\times[a]$ and the vector $\hat F_{[a]}(k,\r)$ is the
restriction of $\hat F(k,\r)$ to $[a]$  --  here we have suppressed the upper index $\eta$.
Denotes by $L(k,[a],\r)$ the Hermitian operator in the left hand side of equation \eqref{homo3-4bis}. We want to estimate the operator norm of $L(k,[a],\r)^{-1}$, i.e. we want to estimate from below the modulus of the eigenvalues of $L(k,[a],\r)$.

Let $\alpha(\r)$ denote an eigenvalue of the matrix $Q_{[a]}(\r)$, $a\in\L$.
It follows from \eqref{ass1} that 
$$|\alpha(\r) -\la(\r) |\leq  \frac{\delta_0} {2}\leq  \frac{c_0} {2}$$
for some appropriate $a\in[a]$, which implies that  
$$|\alpha(\r)|\ge\frac{c_0}2 w_a^{\ga}\ge \ka\ w_a$$ 
by \eqref{laequiv}. Hence, 
$$
\|L(0,[a],\r)^{-1}\| \leq \big(\ka \ w_a\big)^{-1}\quad \forall\r,\ \forall a.
$$
Assume that  $0<|k|\le N$.  Since $|\langle k, \om(\r)\rangle |\lsim N$ it follows from \eqref{laequiv}
 that
$$
|\langle k, \om(\r)\rangle  + \alpha(\r)|\geq \frac 1 2 w_a^{\ga}\ge \ka\ w_a$$
whenever $ w_a\gsim( \frac{N}{c_0})^{\frac1{\ga}}$. Hence  for these $a$'s we get
\be\label{inverse}
\|L(k,[a],\r)^{-1}\| \leq \big(\ka \ w_a\big)^{-1}\quad \forall\r.
\ee
Now  let $ w_a\lsim (\frac{N}{c_0})^{\frac1{\ga}}$. By Assymption A2(ii) we have either
$$|\langle k, \om(\r)\rangle  + \la(\r)|\geq \delta_0 w_a\quad \forall\r, \forall a$$ 
or we have a unit vector ${\mathfrak z}$ such that
$$ \nazz(  \langle k,\om(\r)\rangle+\la(\r)) \geq \delta_0\quad \forall\r, \forall a.$$
 The first case clearly implies  \eqref{inverse},
so let us consider the second case.
By  \eqref{ass1} it follows that\footnote{We use that the operator norm is controlled by the Hilbert Schmidt norm: $\| M\|\leq \| M\|_{HS}$.}
$$ \| \nazz H_{[a]}(\r)   \| \le \frac{\de_0}2.$$
The Hermitian matrix $( \langle k, \om(\r) \rangle  + Q(\r)_{[a]})$ is of dimension $\lsim w_a^{d} $ (see \eqref{block})    therefore, by   Lemma \ref{EK}, we conclude that   \eqref{inverse} holds
for all $\r$ outside   a suitable set $F_{a,k}$ of measure 
$\lsim w_a^{d} \ka\delta_0^{-1}$ . Let
$$
\D_2=\D\setminus  \bigcup_{\substack{|k|\leq N\\  w_a \lsim (\frac{N}{c_0})^{\frac{1}{\ga}}}}F_{a,k}.
$$
Then we get
$$
\meas(\D\setminus \D_2)\lsim 
 N^n\Big (\frac{N}{c_0} \Big)^{\frac{d+1}{\ga}}\frac{\ka}{\delta_0} 
$$ 
and \eqref{inverse} holds for all $\r\in\D_2$, all $|k|\le N$ and all $[a]$.

The equation \eqref{homo3-4bis} is now  solved by
\be\label{homeq3}\hat S_{[a]}(k,\r)= \chi_{|k|\le N} (k)
L(k,[a],\r)^{-1}\hat F_{[a]}(k,\r), \quad  a\in\L\,,
 \ee
and 
\be\label{homeq32}
\hat R_{a}(k,\r)= { \chi_{|k| > N} (k)}
\hat F_{a}(k,\r ), \quad  a\in\L \,.
\ee
Using \eqref{inverse}  we have for $\r\in\D_2$ 
\begin{align*}
\|S_{[a]}(\theta,\r)\|\lsim &\frac{1}{\ka\ w_a (\s-\s')^{n}}\sup_{|\Im\theta|<\s}\| F_{[a]}(\theta,\r)\|,\\
| R_{a}(\theta,\r)|\lsim &\frac{e^{- \frac12(\s-\s')N}}{(\s-\s')^{n}}\sup_{|\Im\theta|<\s}| F_{a}(\theta,\r)|.
\end{align*}
for $\theta\in \T^d_{\s'}\,$, see \eqref{esti}. \\
Since $\|S\|^2_s=\sum_{a\in\L} w_a^{2s}|S_a|^2=\sum_{a\in\hat\L} w_a^{2s}\|S_{[a]}\|^2$ these estimates imply that 
\begin{align*}
\| S(\theta,\r)\|_{s+1}+|S(\theta,\r)|_{\b+}\lsim &\frac{1}{\ka (\s-\s')^{n}}\sup_{|\Im\theta|<\s}\| F(\theta,\r)\|_{s},\\
\| R(\theta,\r)\|_{s}+|R(\theta,\r)|_{\b+}\lsim &\frac{e^{- \frac12(\s-\s')N}}{ (\s-\s')^{n}}\sup_{|\Im\theta|<\s}\| F(\theta,\r)\|_{s},
\end{align*}
for any $\s'\le \s$. 
The estimates of the derivatives with respect to $\r$ are obtained by differentiating
 \eqref{homeq3} and  \eqref{homeq32}   as in Proposition  \ref{prop:homo3}. 
 
 The  functions $F$ and $R$ are complex,  {and a-priori the constructed solution
 $S_\zeta$ also may be complex. Instead of proving that it is real, we replace $S_\zeta, \theta\in\T^n$,
 by its real part and then analytically extend it to $\T^n_{\sigma'}$, using the relation
 $
 \Re S_\zeta(\theta,\rho) :=\frac12 ( S_\zeta(\theta,\rho) +\bar  S_\zeta(\bar\theta,\rho) ).
 $
 Thus we obtain a real solution
 which obeys the same estimates.}
\endproof

\subsection{The last equation}\label{s5.4}
 Concerning the fourth component of the homological equation, \eqref{homo4}, we have the following result
\begin{proposition}\label{prop:homo4}
Let $\om:\D\to\R^n$  be  $\Ca^1$ and verifying 
${|\om-\om_0|_{\Ca^1(\D)}\le\delta_0.}$  Let $\D\ni\r\mapsto A(\r)\in \NF\cap\M_\b$ be $\Ca^1$ and verifying 
\be\label{ass2}
 \left\| \p_\r^j (A(\r)-A_0(\r))_{[a]} \right\|_{HS} \le \frac{\de_0}{4w_a^{2\b}}\ee
for $j=0,1$, $a\in\L$  and $\r\in \D$.
Let $f\in \Tc^{s,\b}(\s,\mu,\D)$, $0<\ka\le\frac{\delta_0}2$ and $N\ge 1$. \\
 Then there exists a subset $\D_3=\D_3(h, \ka,N)\subset \D$, satisfying 
 $$\meas (\D\setminus \D_3)\leq  
 C \Big(\frac{ N}{c_0} \Big)^{\exp} \Big(\frac{\ka}{\de_0} \Big)^{\exp'}, $$
and there exist real $\Ca^1$-functions 
$B: \D_3\to \M_\b \cap \NF $, $\hat S_{\zeta\zeta}(0): \D_3\to \M$ and 
$S_{\zeta\zeta}-\hat S_{\zeta\zeta}(0)$, $R_{\zeta\zeta}(\cdot;\r):\T^n_{\s}\times \D_3\to \M_\b$,
analytic in $\theta$, such that
\be\label{homo4ter} 
\langle  \nabla_\theta S_{\zeta\zeta}(\theta,\r), \om(\r)\rangle  -A(\r)JS_{\zeta\zeta}(\theta,\r)+
S_{\zeta\zeta}(\theta,\r)JA(\r)=-f_{\zeta \zeta} (\theta,\r)+B(\r)+R_{\zeta\zeta}(\theta,\r)
 \ee
and for all $(\theta,\r)\in \T^n_{\s'}\times \D_3$, $\s'<\s$, and $j=0,1$
  \begin{align}
\label{estim-homo4R}
\mu^2\left| \p_\r^j R_{\zeta\zeta}(\theta,\r)\right|_{\b}&\leq  C\ \frac{e^{-\frac12(\s-\s')N}}{ (\s-\s')^{n}}[f]^{s,\b}_{\s,\mu,\D},\\
\label{estim-homo4S}
\mu^2 \left|\p_\r^j  S_{\zeta\zeta}(\theta,\r)\right|_{\b+}&\leq C\
\frac{1}{\kappa^2(\s-\s')^{n}}[f]^{s,\b}_{\s,\mu,\D},\\
 \label{B}
\mu^2 \left|\p_\r^j B(\r)\right|_{\b}&\leq  C  [f]^{s,\b}_{\s,\mu,\D}.\end{align}
The two exponents $\exp$ and $\exp'$ are positive numbers depending on $n$, $\ga$, $d$, $\a_1$, $\a_2$, $\b$. The constant $C$  depends also on $h_0$.
%The exponent $\exp$ only depends on $d, n,\ga$ while the constant $C$ depends on $ |\om_0 |_{\Ca^1(\D)}$,
%$ \| A_0 \|_{\Ca^1(\D)}$. 
\end{proposition}
\proof
As in the previous section, and using the same notation, we re-write \eqref{homo4ter} in complex variables. 
So we introduce 
$S={}^t\! U S_{\zeta,\zeta} U$, $R={}^t\! U R_{\zeta\zeta} U$ and $F={}^t\! U F_{\zeta\zeta} U$.
\\
 By construction, $ S^b_a\in \M_{2\times 2}$ for all $a,b\in\L$. Let us denote  $$ S^b_a= \left( \begin{array}{cc} (S^b_a)^{\xi\xi} & (S^b_a)^{\xi\eta} \\   (S^b_a)^{\xi\eta} & (S^b_a)^{\eta\eta}\end{array}\right) $$ and then $$S^{\xi\xi}=((S^b_a)^{\xi\xi})_{a,b\in\L}, \quad S^{\xi\eta}=((S^b_a)^{\xi\eta})_{a,b\in\L},\quad S^{\eta\eta}=((S^b_a)^{\eta\eta})_{a,b\in\L}.$$ We use similar notations for $R$, $B$ and $F$.\\
In this notation  \eqref{homo4ter}  decouples into three equations
\footnote{Actually \eqref{homo4ter}   decomposes into four scalar equations but the fourth one is the transpose of the third one.}
\begin{align*}
&\langle \nabla_\theta  S^{\xi\xi},\om \rangle +i Q S^{\xi\xi}+ i S^{\xi\xi}\ {}^tQ= B^{\xi\xi}- F^{\xi\xi}+ R^{\xi\xi},\\ 
&\langle \nabla_\theta  S^{\eta\eta}, \om\rangle  -i\ {}^tQ S^{\eta\eta} -i  S^{\eta\eta} Q= B^{\eta\eta}- F^{\eta\eta}+ R^{\eta\eta},\\ 
&\langle \nabla_\theta  S^{\xi\eta},\om\rangle  + iQ S^{\xi\eta} -  iS^{\xi\eta} Q= B^{\xi\eta}- F^{\xi\eta}+ R^{\xi\eta}\,,\end{align*}
where we recall that $Q$ is the scalar valued matrix associated to $A$ via the formula \eqref{Q}. {
The first and the second equations are of the same type, so we focus on the resolution of   the second and the 
third equations.}
Written in Fourier variables,  they read
 \begin{align} %\label{homo4-1}
%&i(\langle k, \om \rangle - Q)\hat S^{\xi\xi}(k) - i\hat S^{\xi\xi}(k)\, {}^tQ= \delta_{k,0}B^{\xi\xi}-\hat F^{\xi\xi}(k)+\hat R^{\xi\xi}(k),\quad k\in \Z^n,\\ 
\label{homo4-2}
&i(\langle k ,\om\rangle -{}^tQ)\hat S^{\eta\eta}(k) - i\hat S^{\eta\eta}(k) Q= \delta_{k,0}B^{\eta\eta}-\hat F^{\eta\eta}(k)+\hat R^{\eta\eta}(k),\quad k\in \Z^n,\\ \label{homo4-3}
&i(\langle  k,\om\rangle +Q)\hat S^{\xi\eta}(k) - i\hat S^{\xi\eta}(k) Q= \delta_{k,0}B^{\xi\eta}-\hat F^{\xi\eta}(k)+\hat R^{\xi\eta}(k),\quad k\in \Z^n   \,,
\end{align}
where  $\delta_{k,j}$ denotes the Kronecker symbol. 

\medskip

{\it Equation \eqref{homo4-2}. } We  { chose $B^{\eta\eta}=0$ and}
 decompose the equation into ``components'' on each product block
$[a]\times[b]$:
\be\label{homo+}
L\, \hat S_{[a]}^{[b]}(k) 
= i\hat F_{[a]}^{[b]}(k,\r){ -i}\hat R_{[a]}^{[b]}(k)
\ee
where we have suppressed the upper index ${\eta\eta}$ and the operator $L:= L(k,{[a]},{[b]},\r)$ is the linear   Hermitian operator, acting in the space of complex
$[a]\times[b]$-matrices defined by
$$
L\, M= \big(
\langle k, \om(\r)\rangle  - {}^tQ_{[a]}(\r)\big) M 
- M Q_{[b]}(\r).$$
 The matrix $Q_{[a]}$ can be diagonalized in an orthonormal basis:
 $${}^tP_{[a]}Q_{[a]}P_{[a]}=D_{[a]}.$$
Therefore denoting $\hat{S'}_{[a]}^{[b]}={}^tP_{[a]}S_{[a]}^{[b]}P_{[b]}$, $\hat{F'}_{[a]}^{[b]}={}^tP_{[a]}F_{[a]}^{[b]}P_{[b]}$ and $\hat{R'}_{[a]}^{[b]}={}^tP_{[a]}R_{[a]}^{[b]}P_{[b]}$ the homological equation \eqref{homo+} reads
\be\label{homo++}( \langle k,\om\rangle +D_{[a]})\hat{S'}_{[a]}^{[b]}(k)-{S'}_{[a]}^{[b]}(k)D_{[b]}=i\hat{F'}_{[a]}^{[b]}(k){ -i}\hat {R'}_{[a]}^{[b]}(k).\ee
This equation can be solved term by term:
\be\label{R'}\hat{R'}_{j\ell}(k)= \hat{F'}_{j\ell}(k),\quad j\in[a],\ \ell\in[b],\ |k|> N\ee
and
\be\label{S'}\hat{S'}_{j\ell}(k)=\frac i{\langle k,\om(\r)\rangle \ -\alpha_j(\r)-\beta_\ell(\r)}\hat{F'}_{j\ell}(k),\quad j\in[a],\ \ell\in[b],\ |k|\leq N\ee
where $\alpha_j(\r)$ and $\beta_\ell(\r)$ denote eigenvalues of $Q_{[a]}(\r)$ and $Q_{[b]}(\r)$, respectively. 
As $Q_{[a]}=\diag\{\la\ : a\in[a]\}+B_{[a]}$ with $B$ Hermitian,  using  hypothesis \eqref{ass2}
 we get that 
\be\label{alpha-a}|(\alpha_j(\r)+\beta_\ell(\r)) -(\la(\r)+\lb(\r))|\leq
 \frac{\de_0}{4}+ \frac{\de_0}{4}\le \frac{\de_0} {2}.
 \ee
It follows as in the proof of Proposition \ref{prop:homo3}, using Lemma~\ref{EK}, relation
\eqref{laequiv}, Assumption A2(iii)
and \eqref{ass2},  that there exists a subset 
$\D_2=\D_2(h,\ka,N)\subset \D$,  satisfying
 $$\meas (\D\setminus \D_2)\leq  C\Big(\frac{ N}{c_0}\Big)^{\exp} \frac{\ka}{\de_0}, $$
such that 
$$| \langle k,\om(\r)\rangle \ -\alpha_j(\r)-\beta_\ell(\r)|\geq {\ka}(1+|w_a+w_b|),$$
holds for all $\r\in\D_2$, all $|k|\le N$, all $j\in[a],\ \ell\in[b]$ and all $[a],[b]\in\hat\L$.
Thus for $\r\in\D_2$ we obtain
$$|\hat S'_{j\ell}|=\frac 1{ \ka (1+|w_a+w_b|)}|\hat F'_{j\ell}|$$
which leads to 
\begin{align*}\left\| S_{[a]}^{[b]} \right\|_{HS}=\left\| {S'}_{[a]}^{[b]} \right\|_{HS}&\leq \frac 1{ \ka (1+|w_a+w_b|)}\left\| {F'}_{[a]}^{[b]} \right\|_{HS}\\
&=\frac 1{ \ka (1+|w_a+w_b|)}\left\| F_{[a]}^{[b]} \right\|_{HS}.\end{align*}
Therefore we obtain a solution  satisfying for any $|\Im\theta|<\s'$
\begin{align*}
| S(\theta)|_{\b+}\lsim &\frac{1}{\ka (\s-\s')^{n}}
\sup_{|\Im\theta|<\s}| F(\theta)|_{\b},\\
|R(\theta)|_{\b}\lsim &\frac{e^{-\frac12 (\s-\s')N}}{ (\s-\s')^{n}}
\sup_{|\Im\theta|<\s}| F(\theta)|_{\b}.
\end{align*}

The estimates for the derivatives with respect to $\r$ are obtained by differentiating (\ref{S'}) 
and (\ref{R'}).

\medskip

{\it Equation \eqref{homo4-3}. } 
It remains to consider  \eqref{homo4-3} which decomposes into the ``components''
over the product blocks $[a]\times[b]$:
\be\label{homo-}
\begin{split}
L\, \hat S_{[a]}^{[b]}(k)& :=
 \langle k, \om(\r)\rangle \ 
\hat S_{[a]}^{[b]}(k) +Q_{[a]}(\r)\hat S_{[a]}^{[b]}(k) \\&- \hat S_{[a]}^{[b]}(k) Q_{[b]}(\r)= 
-i\delta_{k,0}B_{[a]}^{[b]}+i \hat F_{[a]}^{[b]}(k,\r)-i\hat R_{[a]}^{[b]}(k),
\end{split}
\ee
--  here $L=L(k,[a],[b],\rho)$, and 
 we have suppressed the upper index ${\xi\eta}$. We use the notation from the study
of equation \eqref{homo4-2} above and we assume without loss of generality that $w_a\leq w_b$.
Now  $L(k,[a],[b],\rho)$ is a linear operator acting in the space of complex  $[a]\times[b]$-matrices. Its eigenvalues are 
\be\label{divisors}
\langle k,\omega(\rho)\rangle +\alpha_j(\rho)-\beta_\ell(\rho),\quad j\in [a],\ \ell\in [b].
\ee
To estimates these eigenvalues we have to distinguish two cases, depending on whether $k= 0$ or not.

\medskip

{\it The case $k=0$.}
In this case   we distinguish  whether $w_a=w_b$ or not.

\medskip

\noindent When $w_a\neq w_b$, we use   \eqref{ass2} and \eqref{la-lb} to get 
$$|\alpha_j(\r)-\beta_\ell(\r)|\geq c_0|w_a-w_b|-\frac{\de_0}{4w_a^{2\b}}
-\frac{\de_0}{4w_b^{2\b}}\geq \ka|w_a-w_b|.$$
This last estimate allows us to solve \eqref{homo-}, choosing 
$$B_{[a]}^{[b]}=\hat R_{[a]}^{[b]}(0) =0$$
and
$$\hat S_{[a]}^{[b]}(0)= L(0,[a],[b],\r)^{-1}\hat F_{[a]}^{[b]}(0)$$
with 
$$
\left\|\hat S_{[a]}^{[b]}(0)\right\|_{HS}\lsim  \frac{1}{c_0|w_a-w_b|}\left\|\hat F_{[a]}^{[b]}(0)\right\|_{HS}.
$$
This implies that 
$$
| \hat S(0)|_{\b+}\lsim \frac{1}{\ka}| \hat F(0)|_{\b},$$
and the estimates of the derivatives with respect to $\r$ are obtained by differentiating the expression for
$\hat S_{[a]}^{[b]}(0)$.

\medskip

\noindent When $w_a=w_b$, we cannot control $|\alpha_j(\r)-\b_\ell(\r)|$ from below, 
so  we define
$$\hat S_{[a]}^{[b]}(0)=0,\quad \hat R_{[a]}^{[b]}(0)=0
$$
and
$$
B_{[a]}^{[b]}=\hat F_{[a]}^{[b]}(0).$$
This gives the estimates
$$
|B|_{\b}\leq | \hat F(0)|_{\b}.$$
The estimates of the derivatives with respect to $\r$ are obtained by differentiating the expressions for
$B$.

\medskip

{\it The case $k\not=0$. }  If $k\neq 0$ we face the small divisors \eqref{divisors} with 
non-trivial $\langle k,\omega\rangle$. 
 Using Hypothesis A3, there is a set
 $\D_2'=\D(\om, 2\eta, N)$, 
 $$\meas(\D\setminus {\D_2'})\lsim N^{\a_1}(\frac{\eta}{\delta_0})^{\a_2} ,$$
 such that for all $\r\in\D_2'$ and  $0<|k|\leq N$ 
  $$|\langle k, \om(\r)\rangle \ -\la(\r)+\lb(\r)|\geq 2\eta(1+|w_a-w_b|).$$
 By \eqref{ass2} this implies
 \begin{align*}|\langle k, \om(\r)\rangle \ -\alpha_j(\r)+\beta_\ell(\r)|&\geq 2\eta(1+|w_a-w_b|)-\frac{\de_0}{4w_a^{2\b}}
 -\frac{\de_0}{4w_b^{2\b}}\\
& \geq \eta(1+|w_a-w_b|)\end{align*}
 if
$$ w_b\geq w_a\geq \Big( \frac{\de_0}{2\eta}\Big)^{\frac1{2\b}}.$$
Let now $ w_a\lsim  ( \frac{\de_0}{2\eta})^{\frac1{2\b}}$. We note that $|\langle k, \om(\r)\rangle \ -\la(\r)+\lb(\r)|\leq 1$ implies that 
$w_b\lsim( \frac{\de_0}{2\eta})^{\frac1{2\b}}+ C|k|\lsim ( \frac{\de_0}{2\eta})^{\frac1{2\b}}+ N$.
Using Assumption A2 (iv) and condition \eqref{ass2} we get 
as in section 5.3 that 
\be\label{inversebis}
 |\langle k,\omega(\rho)\rangle +\alpha_j(\rho)-\beta_\ell(\rho)|\geq {\ka}(1+|w_a-w_b|)\quad \forall j\in[a],\ \forall \ell\in[b]\ee
 holds
outside a set $F_{[a],[b],k}$ of measure $\lsim w_a^dw_b^{d}(1+|w_a-w_b|)\ka\de_0^{-1}$.

If $F$ is the union of  $F_{[a],[b],k}$ for $|k|\leq N$, $[a],[b]\in\hat\L$ such that $ w_a\lsim  ( \frac{\de_0}{2\eta})^{\frac1{2\b}}$ and $w_b \lsim ( \frac{\de_0}{2\eta})^{\frac1{2\b}}+ N$ respectively, we have 
 \begin{align*}\meas(F)&\lsim
( \frac{\de_0}{2\eta})^{\frac1{2\b}}\big(( \frac{\de_0}{2\eta})^{\frac1{2\b}}+ N\big)^{d+1}\frac\ka{\de_0}N^n\big(( \frac{\de_0}{2\eta})^{\frac1{2\b}}+ N\big)( \frac{\de_0}{2\eta})^{\frac{1}{2\b}}\\
&\lsim N^{n+d+2} ( \frac{\de_0}{\eta})^{\frac{4+d}{2\b}}\frac\ka{\de_0}\,.
\end{align*}
Now we choose $\eta$ so that
$$(\frac{\eta}{\delta_0})^{\a_2}=( \frac{\de_0}{\eta})^{\frac{4+d}{2\b}}\frac\ka{\de_0}\quad \text{i.e. }\frac\eta{\de_0}=\big(\frac\ka{\de_0}\big)^{\frac{2\b}{4+d+2\b \a_2}}.
$$
Then, as $\b\leq 1$, $\eta\leq \ka$ and we have
$$\meas(F)\lsim N^{n+d+2} \big(\frac\ka{\de_0}\big)^{\frac{2\b\a_2}{2+d+2\b\a_2}}\,.$$
Let $\D_3=\D_2\cap\D_2'\setminus F$, we have
$$\meas(\D\setminus\D_3)\lsim N^{exp}\big(\frac\ka{\de_0}\big)^{\frac{2\b\a_2}{2+d+2\b\a_2}}$$
and by construction  for all $\r\in\D_3$,
$0<|k|\le N$, $a,b\in\L$ and $j\in[a],\ \ell\in[b]$ we have
$$|\langle k, \om(\r)\rangle \ -\alpha_j(\r)+\beta_\ell(\r)| \geq \ka(1+|w_a-w_b|).$$
 Hence, following the same procedure of diagonalization of $L$ as in the resolution of equation \eqref{homo4-2},   the solution
$$\hat S_{[a]}^{[b]}(k)=\chi_{0<|k|\le N}\,
 L(k,[a],[b],\r)^{-1}\hat F_{[a]}^{[b]}(k),
 $$
and
$$\hat R_{a}^{b}(k)=\chi_{|k|> N}\,
 \hat F_{a}^{b}(k)\,,
 $$
satisfies
\begin{align*}
| S(\theta) |_{\b+}\lsim &\frac{1}{\ka (\s-\s')^{n}}
\sup_{|\Im\theta|<\s}| F(\theta)|_{\b},\\
|R|_{\b}\lsim &\frac{e^{-\frac12(\s-\s')N}}{ (\s-\s')^{n}}
\sup_{|\Im\theta|<\s}| F(\theta)|_{\b}.
\end{align*}
for $\theta\in \T^d_{\s'}$.
The estimates of the derivatives with respect to $\r$ are obtained by differentiating the expressions
for $S$ and $R$.

In this way we have constructed a solution $S_{\zeta \zeta}, R_{\zeta \zeta}, B$  of 
the fourth component of the homological equation which satisfies all required estimates.
To guarantee that it is real, as at the end of Section~\ref{s5.3} we replace  $S_{\zeta \zeta}, R_{\zeta \zeta}, B$
 by their real parts
(i.e., replace $S_{\zeta \zeta}(\theta,\rho)$ by $\frac12(S_{\zeta \zeta}(\theta,\rho) + 
\bar S_{\zeta \zeta}( \bar\theta,\rho) )$, etc.) 
 \endproof

\subsection{Summing up}

Let
$$h=\om(\r)\cdot r +\frac 1 2 \langle \zeta,  A(\r)\zeta\rangle$$
where $\r\to\om(\r)$ and $\r\to A(\r)$ are $C^1$ on $\D$ and $A$ is on normal form.
\begin{proposition}\label{thm-homo}

Assume
\be\label{Aom}
 |\partial_\r^j (A(\r)-A_0(\r))|_\b \le \frac{\delta_0}{4},\quad |\partial_\r^j(\om-\om_0)|\leq \delta_0\ee
for $j=0,1$  and $\r\in \D$.
Let $f\in \Tc^{s,\b}(\s,\mu,\D)$, $0<\ka\leq \frac{\de_0}2$ and $N\ge 1$. 
 Then there exists a subset $\D'=\D'(h, \ka,N)\subset \D$, satisfying
 $$\meas (\D\setminus \D')\leq  
 CN^{\exp} \Big(\frac{\ka}{\delta_0}\Big)^{\exp'}, $$
 and there exist real jet-functions 
   $S\in\Tc^{s,\b+}(\s',\mu,\D')$ , $R\in\Tc^{s,\b}(\s',\mu,\D')$ and a  normal form
$$\hat h=\( f(\cdot,0;\r) \)+\( \nabla_r f(\cdot,0;\r) \)\cdot r+ \frac 1 2 \langle \zeta, B(\r)\zeta\rangle,$$
  such that
$$
\{ h,S \}+f^T=\hat h+R.
$$
Furthermore, for all  $0\le\s'<\s$
\be\label{estim-B}|\partial_\r^jB(\r)|_{\b}\le
 [f]^{s,\b}_{\s,\mu,\D'},\quad j=0,1 \text{ and }\r\in\D'
\ee
\be\label{estim-S}
[S ]^{s,\b+}_{\s',\mu,\D'} \leq 
C\frac{1}{\ka^2 (\s-\s')^{n}}[f]^{s,\b}_{\s,\mu,\D'}\ee
 
\be\label{estim-R}
[R]^{s,\b}_{\s',\mu,\D'}\leq 
C
\frac{e^{-\frac12(\s-\s')N}}{ (\s-\s')^{n}}[f]^{s,\b}_{\s,\mu,\D'}.\ee
The two exponents $\exp$ and $\exp'$ are positive numbers depending on $n$, $d$, $\a_1$, $a_2$, $\ga$, $\b$. The constant $C$  depends also on $h_0$.
 
\end{proposition}

\section{Proof of the KAM Theorem.}
The theorem \ref{main} is proved by an iterative KAM procedure. We first describe the general step of this KAM procedure.
\subsection{The KAM step}
Let $h$ be a normal form Hamiltonian
$$
h= \om\cdot r +\frac 1 2 \langle \zeta, A(\om)\zeta\rangle$$
with $A$ on normal form, $A-A_0\in\M_\b$ and satisfying \eqref{Aom}. Let $f\in \Tc^{s,\b}(\D,\s,\mu)$ be a (small) Hamiltonian perturbation.  
Let $S=S^T \in \Tc^{s,\b+}(\D',\s',\mu)$ be the  solution of the homological equation
\be \label{eq-homobis}
\{ h,S \}+f^T=\hat h+R.
\ee
defined in Proposition \ref{thm-homo}.
Then defining 
 $$ h^+:=h+\hat h,$$
we get
$$h\circ \Phi^1_S=h^++f^+$$
with
\be \label{f+}
f^+= R+(f-f^T)\circ \Phi^1_S+\int_0^1\{ (1-t)(\hat h+R)+tf^T,S \}\circ \Phi^t_S\ \dd t.
\ee
The following Lemma gives an estimation of the new perturbation:

\begin{lemma}\label{lem-f+}
Let $\ka>0$, $N\geq 1$, $0<\s'<\s\leq 1$ and $0<2\mu'< \mu\leq 1$. Assume that $\D'\subset \D$, that $f\in\Tc^{s,\b}(\D,\s,\mu)$, that $R$ satisfies \eqref{estim-R}  and that $S=S^T$ belongs to $\Tc^{s,\b+}(\D',\s'',\mu)$ with $\s''=\frac{\s+\s'}{2}$ and satisfies
\begin{equation}
\label{hypo-S} 
[S]^{s,\b+}_{\D',\s'',\mu}\leq \frac 1{16} \mu^2 ( \s- \s').
\end{equation}
Then the function $f^+$ given by formula \eqref{f+} belongs to $ \Tc^{s,\b}(\D',\s',\mu')$ and
\begin{align}\begin{split}\label{estim-f+}
[f^+]^{s,\b}_{\D',\s',\mu'}\leq C\left(  \frac{e^{-\frac 1 2(\s-\s')N}}{  (\s-\s')^{n}}+ \left( \frac{\mu'}{\mu}\right)^3 \right. 
+\left.  \frac{1}{\ka^2\mu^2 (\s-\s')^{n+1}}[f]^{s,\b}_{\D,\s,\mu}\right) [f]^{s,\b}_{\D,\s,\mu}
\end{split}\end{align}
where $C$  depends on $h_0$.
\end{lemma}
\proof
Let us denote the three terms in the r.h.s. of \eqref{f+} by $f^+_1$, $f^+_2$ and $f^+_3$. In view of \eqref{estim-R}, we have that $[f^+_1]^{\ga',D}_{\Om',\s',\mu'}$ is controlled by the first term in r.h.s. of \eqref{estim-f+}. \\
By Lemma \ref{lemma:jet}, we get 
$$[f-f^T]^{\ga',D}_{\Om',\s,2\mu'}\leq C \left( \frac{\mu'}{\mu}\right)^3 [f]^{\ga,D}_{\Om,\s,\mu}.$$
By hypotheis $S=S^T $ belongs to $ \Tc^{s,\b+}(\D',\s',\mu)$ and satisfies \eqref{hypo-S} which implies $[S]^{s,\b+}_{\D',\s'',\mu}\leq \frac 1{2} (\mu-\mu')^2 ( \s''- \s')$ since $2\mu'<\mu$. Therefore 
 by Proposition \ref{composition} and since $2\mu'\leq 2 (\mu-\mu')$, $[f^+_2]^{s,\b}_{\D',\s',\mu'}$ is controlled by the second term in r.h.s. of \eqref{estim-f+}. 
 
 \smallskip
 
It remains to control $[f^+_3]^{s,\b}_{\D',\s',\mu'}$. To begin with, $g_t:=(1-t)(\hat h+R)+tf^T$ is jet function in $ \Tc^{s,\b}(\D,\s',\mu)$. Furthermore, defining  for $j=1,2$,
$$\s_j=\s'+j\frac{\s-\s'}{3} $$
 and  using \eqref{estim-R} we get (for $N$ large enough)
$$
[g_t]^{s,\b}_{\D',\s_2,\mu}\leq C\left( 1+3^{n}\frac{e^{-(\s-\s')N/6}}{  (\s-\s')^{n}}\right)[f]^{s,\b}_{\D,\s,\mu}\leq C [f]^{s,\b}_{\D,\s,\mu}.
$$
On the other hand $ S\in \Tc^{s,\b+}(\D',\s_2,\mu)$ is also a jet-function and satisfies 
$$
[S]^{s,\b+}_{\D',\s_2,\mu}\leq \frac{C}{\ka^2 (\s-\s')^{n}}[f]^{s,\b}_{\D,\s,\mu}.
$$
Then using Lemma \ref{lemma-poisson} we have 
$$
[\{g_t,S\}]^{s,\b}_{\D',\s_1,\mu}\leq C\frac{1}{\ka^2\mu^2 (\s-\s')^{n+1}}([f]^{s,\b}_{\D,\s,\mu})^2.$$
We conclude the proof by Proposition \ref{composition}.
\endproof
\subsection{Choice of parameters}
To prove the main theorem we  construct the transformation $\Phi$ as the composition of infinitely-many transformations $S$ as in Theorem \ref{thm-homo}, i.e. for all $k\geq 1$ we construct iteratively $S_{k-1}$, $h_k$, $f_k$ following the general scheme \eqref{eq-homobis}--\eqref{f+} in such way
$$(h+f)\circ \Phi^1_{S_{k-1}}\circ\cdots\circ \Phi^1_{S_0}= h_{k}+f_{k}.$$
At each step $f_k\in \Tc^{s,\b}(\D_k,\s_k,\mu_k)$ with $[f_k]^{s,\b}_{\D_k,\s_k,\mu_k}\leq \eps_k$ , $h_k=\lan\om_k,r\ran +\frac 12\lan \zeta,A_k\zeta\ran$ is on normal form, the Fourier series are truncated at order $N_k$ and the small divisors are controlled by $\ka_k$. In this section we specify the choice of all the parameters for $k\geq 1$. \\
First we fix
$$\delta_0= \eps^{1/4},\quad \ka_0=\eps^{1/3}.$$
 We define  $\eps_0=\eps$, $\s_0=\s$,  $\mu_0=\mu$ and for $j\geq 1$ we choose
\begin{align*}
\s_{j-1}-\s_j=&C_* \s_0 j^{-2},\\
\mu_j=&\eps_{j-1}^{2/5}\mu_0,\\
N_j=&(\s_{j-1}-\s_j)^{-1}\ln \eps_j^{-1},\\
\ka_{j-1}=&\eps_{j-1}^{\frac 1 {64}} 
\end{align*}
where $(C_*)^{-1} =2\sum_{j\geq 1}\frac 1{j^2}$.\\
The numbers above are defined in terms of $\eps_j$'s which are defined inductively (with given $\eps=\eps_0$) accordingly to Lemma \ref{lem-f+} through the relation 
\begin{align}\begin{split}\label{eps+}
\eps_{j+1}= &C\left( \frac 12\eps_j (j+1)^{2n}\s_0^{-n}+\left(\frac{\eps_{j}}{\eps_{j-1}}\right)^\frac{6}{5}\right.  \\
&+ \left. (j+1)^{2(n+1)}\s_0^{-n-1}\mu_0^{-2}  \eps_j^{1/5- 1/ {32}} \right)\eps_j.
\end{split}\end{align}
\begin{lemma}\label{eps}
For all $j\geq 1$
$$
\eps_j\leq\eps_0^{(7/6)^j}$$
provided that $\eps=\eps_0>0$ is sufficiently small (in terms of $n$, $\s_0$, $\mu_0$).
%Furthermore, if at each step $S_k$ is constructed as in Proposition \ref{thm-homo} and thus satisfies \eqref{estim-S}, assumption \eqref{hypok1} holds for each $k$.
\end{lemma}
\proof
It suffices to check that if
$$
\eps_k\leq\eps_{k-1}^{7/6}\quad \text{ for all } k\leq j,$$
then all the three terms in the r.h.s. of \eqref{eps+} are $\leq \frac 1 3 \eps_j^{7/6}$. For the first term this is   straightforward assuming $\eps$ small enough. To obtain the same estimate for the third term it suffices to notice that $1/6<1/5-1/32$. Concerning the second term, we have $\eps_{j}\leq \eps_j^{1/7}\eps_{j-1}$ and thus
$$C\left(\frac{\eps_{j}}{\eps_{j-1}}\right)^\frac{6}{5}\leq C\left(\eps_j^\frac{1}{7} \right)^\frac{6}{5}\leq C \eps_j^{\frac 6 {35}-\frac 6{36}}\eps_j^\frac{1}{6}\leq \frac 1 3 \eps_j^\frac{1}{6}$$
for $\eps$ small enough.
%The estimate \eqref{hypok1} is a direct consequence of estimation \eqref{estim-S} and of the choice of $\mu_j$, $\ka_j$ and $\s_j$.
\endproof

\subsection{Iterative lemma}

Let set $\D_0=\D$, $h_0=\lan\om_0(\r), r\ran+\frac 1 2 \langle \zeta, A_0(\r)\zeta\rangle$ and  $f_0= f$ in such a way $[f_0]^{s,\b}_{\D_0,\s_0,\mu_0}\leq \eps_0$. For $k\geq 0$ let us denote
$$\O_{k}= \O^{s}(\s_{k},\mu_{k}).$$
\begin{lemma}\label{iterative} For $\eps$ sufficiently small depending on $\mu_0$, $\s_0$, $n$,$s$, $\b$ and $h_0$ we have the following:\\
For all $k\geq 1$ there exist $\D_k\subset\D_{k-1}$, $S_{k-1}\in \Tc^{s,\b+}(\D_k,\s_k,\mu_k)$, $h_k=\lan\om_k,r\ran +\frac 12\lan \zeta,A_k\zeta\ran$  on normal form and 
$f_k\in \Tc^{s,\b}(\D_k,\s_k,\mu_k)$ such that
\begin{itemize}
\item[(i)]  The mapping \be \label{Phik} \Phi_{k}(\cdot,\r)=\Phi^1_{S_{k-1}}\ :\ \O(k)\to \O(k-1), \quad \r\in \D_{k},\ k=1,2,\cdots\ee
is an analytic symplectomorphism linking the hamiltonian at step $k-1$ and the hamiltonian at the step  k, i.e.
$$(h_{k-1}+f_{k-1})\circ \Phi_{k}= h_k+f_k.$$
\item[(ii)] we have the estimates
\begin{align*}
\meas(\D_{k-1}\setminus \D_{k})&\leq \eps_{k-1}^\a,\\
[h_k-h_{k-1}]^{s,\b}_{\D_k,\s_k,\mu_k}&\leq C\eps_{k-1},\\
[f_k]^{s,\b}_{\D_k,\s_k,\mu_k}&\leq \eps_k,\\
\| \Phi_k(x,\r)-x\|_s&\leq \eps_{k-1}^{1/6},\ \text{ for } x\in \O(k),\ \r\in\D_k.
\end{align*}
\end{itemize}
The  exponents $\a$ is a positive number depending on $n$, $d$, $\a_1$, $a_2$, $\ga$, $\b$. The constant $C$  depends also on $h_0$..
 \end{lemma}
 \proof At step 1, $h_0=\lan\om_0(\r), r\ran+\frac 1 2 \langle \zeta, A_0(\r)\zeta\rangle$ and thus hypothesis \eqref{Aom} is trivially satisfied and we can apply Proposition \ref{thm-homo} to construct $S_0$, $R_0$, $B_0$ and $\D_1$ such that for $\r\in\D_1$
 $$\{h_0,S_0\}+f_0^T=\hat h_0+R_0.$$
Then we see that, using \eqref{estim-S} and defining  $s_{1/2}=\frac{\s_0+\s_1}2$, we have
$$ [S_{0}]^{s,\b+}_{\D_{1},\s_{1/2},\mu_0}\leq C \frac{\eps_0}{\ka_0^2 (\s_0-\s_{1/2})^n}\leq\frac 1{16} \mu_0^2 ( \s_0- \s_{1})$$
for $\eps=\eps_0$ small enough in view of our choice of parameters.
 Therefore both Proposition \ref{Summarize} and Lemma \ref{f+} apply  and thus for any $\r\in\D_1$,
 $\Phi_{1}(\cdot,\r)=\Phi^1_{S_{0}}\ :\ \O(1)\to \O(0)$ is an analytic symplectomorphism such that
 $$(h_{0}+f_{0})\circ \Phi_{1}= h_1+f_1$$ with  $h_1$, $f_1$, $\D_1$ and $\Phi_1$  satisfying the estimates (ii$)_{k=1}$ . In particular we have
 $$\|\Phi_1(x)-x\|_s\leq \frac C{\s_0\mu_0^2}[S_{0}]^{s,\b+}_{\D_{1},\s_{1/2},\mu_0}\leq \frac C{\s_0^{n+2}\mu_0^2\ka_0^2}\eps_0\leq \frac C{\s_0^{n+2}\mu_0^2}\eps_0^{1/3}\leq \eps_0^{1/6}     $$
 for $\eps_0$ small enough. 
 
 \medskip
 
 Now assume that we have completed the iteration up to step j.
We want to perform the step $j+1$. We first note that by construction (see Proposition \ref{thm-homo})
$$A_j =A_0+B_0+\cdots+B_{j-1}$$
and by \eqref{estim-B}
$$|A_j|_\b \leq \eps_0+\cdots +\eps_{j-1}\leq 2\eps_0\leq \frac 14 \delta_0$$ 
for $\eps_0$ small enough.
Similarly
$$\om_j=\om_0+\( \nabla_r f_0(\cdot,0;\r) \)+\cdots +\( \nabla_r f_{j-1}(\cdot,0;\r) \)$$
and thus $|\partial_r^j(\om_j-\om_0)|\leq \delta_0$ for $\eps_0$ small enough.\\ Therefore \eqref{Aom} is satisfied at rank j and we can apply
Proposition \ref{thm-homo} in order to construct $S_j$, $B_j$, $R_j$ and $\D_j$.  

Then we construct  $f_{j+1}$ as in \eqref{f+}, i.e.
$$
f_{j+1}= R_j+(f_j-f_j^T)\circ \Phi^1_{S_j}+\int_0^1\{ (1-t)(\hat h_j+R_j)+tf_j^T,S_j \}\circ \Phi^t_{S_j}\ \dd t.
$$
To control $f_{j+1}$ we  apply Lemma \ref{lem-f+} what we can do since, defining $\s_{j+1/2}=\frac{\s_j+\s_{j+1}}2$,
$$
[S_{j}]^{s,\b+}_{\D_{j+1},\s_{j+1/2},\mu_j}\leq  C \frac{\eps_j}{\ka_j^2 (\s_j-\s_{j+1})^n}  \leq \frac 1{8} \mu_j^2 ( \s_j- \s_{j+1}).$$
Therefore combining Lemma \ref{lem-f+} and  \eqref{eps+} we conclude that
$$[f_{j+1}]^{s,\b}_{\D_{j+1},\s_{j+1},\mu_{j+1}}\leq \eps_{j+1}.$$
On the other hand by Proposition \ref{thm-homo} the domain $\D_{j+1}$ satisfies
$$\meas (\D_{j}\setminus\D_{j+1})\leq CN^{\exp}_{j}\Big(\frac{\ka_j}{\delta_0}\Big)^{\exp'}\leq \eps_{j}^\alpha $$
 for some $\a>0$ and for $\eps_0=\eps$ small enough.

\endproof

\subsection{Transition to the limit and proof of Theorem \ref{main}}

Let $$\D'=\cap_{k\geq 0}\D_k.$$ In view of the iterative lemma, this is a borel set satisfying
$$\meas(\D\setminus\D')\leq 2 \eps ^\a.$$
Let us set
$$Q_l=\O^{s}(\s/\ell,\mu/\ell), \ \mathcal Z_s=\T_\s^n\times\C^n\times Y_s$$
where $\ell\geq 2$, and recall that  $\|\cdot\|_s$ denotes the natural norm on $\C^n\times\C^n\times Y_s$. It defines the distance on $\mathcal Z_s$. We used the notations introduced in Lemma \ref{iterative}. By Proposition \ref{changevar} assertion 2, for each $\r\in \D'$, the map $\Phi_k$ extends to $Q_2$ and satisfies on $Q_2$ the same estimate as on $\O_k$:
\be \label{estim-Phik}\Phi_k:\ Q_2\to \mathcal Z_s,\quad \|\Phi_k-id\|_s\leq C\mu_k^{-2}(\s_{k-1}-\s_k)^{-1}\eps_k\leq \eps_k^{1/6}.\ee
Now for $0\leq j\leq N$ let us denote $\Phi^j_N=\Phi_{j+1}\circ\cdots\circ \Phi_N$. Due to \eqref{Phik}, it maps $\O(N)$ to $\O(j-1)$. Due to \eqref{estim-Phik}, this map extends analytically to a map $\Phi^j_N:\ Q_3\to \mathcal Z_s,$ and for $M>N$, $\|\Phi_N^j-\Phi_M^j\|_s\leq C \eps_N^{1/6} $, i.e. $(\Phi_N^j)_N$ is a Cauchy sequence. Thus when $N\to \infty$ the maps  $\Phi^j_N$ converge to a limiting mapping $\Phi_\infty^j:\ Q_3\to \mathcal Z_s.$ Furthermore we have
\be \label{estim-Phiinf}\|\Phi_\infty^j-id\|_s\leq C\sum_{k\geq j} \eps_k^{1/6}\leq C\eps_j^{1/6},\ \forall j\geq 1.\ee
By the Cauchy estimate the linearized map satisfies
\be \label{linearized} \|D\Phi_\infty^j(x)-id\|_{\L(Y_s,Y_s)}\leq C\eps_j^{1/6},\quad \forall x\in Q_4,\ \forall j\geq 1.\ee
By construction, the map $\Phi_N0$ transforms the original hamiltonian $$H_0=\lan\om, r\ran+\frac 1 2 \langle \zeta, A_0(\om)\zeta\rangle+f$$ to $$H_N=\lan\om_N, r\ran+\frac 1 2 \langle \zeta, A_N(\om)\zeta\rangle+f_N.$$
Here $$\om_N=\om+\( \nabla_r f_0(\cdot,0;\r) \)+\cdots +\( \nabla_r f_{N-1}(\cdot,0;\r) \)$$ and $$A_N= A_0+B_0+\cdots+B_{N-1}$$ where 
$B_k$ is built from $\langle\nabla^2_{\zeta\zeta}f_k(\cdot,0)\rangle$ as in the proof of Proposition \ref{prop:homo4}.\\
Clearly, $\om_N\to \om'$ and $A_N\to A$ where the vector $\om'\equiv\om'(\r)$ and the operator $A\equiv A(\r)$ satisfy the assertions of Theorem \ref{main}.\\
Let us denote $\Phi=\Phi_\infty^0$, consider the limiting hamiltonian $H'=H_0\circ \Phi$ and write it as
$$ H'= \lan\om', r\ran+\frac 1 2 \langle \zeta, A(\r)\zeta\rangle+f'.$$
The function $f'$ is analytic in the domain $Q_3$. Since $H'=H_k\circ \Phi_\infty^k$,  we have
$$\nabla H'(x)=D\Phi_\infty^k(x)\cdot\nabla H_k(\Phi_\infty^k(x)).$$
As $[f_k]^{s,\b}_{\D_k,\s_k,\mu_k}\leq \eps_k$, we deduce 
 $$\nabla_r H_k(\Phi_\infty^k(\theta,0,0))=\om_k+O(\eps_k^{1/5})\quad \theta\in\T^n_{\frac \s 2}.$$ 
 Since the map $\Phi_\infty^k$ satisfies \eqref{linearized}, then $$\nabla_r H'(\theta,0,0)=\om'+O(\eps_k^{1/6}) \quad \text{ for all } k\geq 1\text{ and } \theta\in\T^n_{\frac \s 2} .$$ Hence, $\nabla_r H'(\theta,0,0)=\om'$ and thus
$$\nabla_r f'(\theta,0,0)\equiv 0\quad \text{ for } \theta\in\T^n_{\frac \s 2} .$$
Similar arguments leads to
$$\nabla_{\zeta_a} f'(\theta,0,0)\equiv 0 \text{ and }\nabla_{\zeta_a}\nabla_r f'(\theta,0,0)\equiv 0\quad \text{ for } \theta\in\T^n_{\frac \s 2} .$$
Now consider $\nabla_{\zeta_a}\nabla_{\zeta_b}H'(x)$. To study this matrix let us write it in the form \eqref{feo}, with $h=H_k$ and $x(1)=\Phi_\infty^k(x)$. Repeating the arguments used in the proof of Proposition \ref{composition} we get that
$$\nabla_{\zeta_a}\nabla_{\zeta_b}H'(\theta,0,0)=(A_k)_{ab}+0(\eps_k^{1/6}) \quad \text{ for all } k\geq 1\text{ and } \theta\in\T^n_{\frac \s 2} .$$
Therefore $\nabla_{\zeta_a}\nabla_{\zeta_b}H'(\theta,0,0)= A_{ab}$ i.e. 
$$\nabla_{\zeta_a}\nabla_{\zeta_b}f'(\theta,0,0)=0 \quad \text{ for } \theta\in\T^n_{\frac \s 2} .$$

This concludes the proof of Theorem \ref{main}.

\end{document}